\documentclass[11pt, a4paper]{amsart}
\usepackage{fullpage}
\usepackage{amssymb}
\usepackage{commath}
\usepackage{empheq}
\usepackage{amsmath, amsthm}
 \usepackage[foot]{amsaddr}
\usepackage{dsfont}
\usepackage{multicol}
\usepackage{comment}

\usepackage{etoolbox}
\makeatletter
\let\ams@starttoc\@starttoc
\makeatother
\usepackage[parfill]{parskip}
\makeatletter
\let\@starttoc\ams@starttoc
\patchcmd{\@starttoc}{\makeatletter}{\makeatletter\parskip\z@}{}{}
\makeatother

\usepackage{tikz}
\usetikzlibrary{matrix,arrows,decorations.pathmorphing,decorations.markings}
\usepackage{tikz-cd}

\usepackage{capt-of}

\usepackage{bbm}

\usepackage{enumitem}

\usepackage{mathrsfs}

\usepackage{url}

\newcommand{\llb}{\left\lbrace}
\newcommand{\rrb}{\right\rbrace}

\renewcommand{\phi}{\varphi}

\newtheorem{thm}{Theorem}[section]

\theoremstyle{plain}
\newtheorem{prop}[thm]{Proposition}
\newtheorem{lem}[thm]{Lemma}
\newtheorem{cor}[thm]{Corollary}

\theoremstyle{definition}
\newtheorem{defn}[thm]{Definition}
\theoremstyle{remark}
\newtheorem{rem}[thm]{Remark}

\newtheorem{eg}[thm]{Example}

\title{Generalized rook-Brauer algebras and their homology}
\author{Daniel Graves}
\address{Lifelong Learning Centre, University of Leeds, Woodhouse, Leeds, LS2 9JT, UK}
\email{dan.graves92@gmail.com}
\date{}

\begin{document}



\maketitle

\begin{abstract}
Rook-Brauer algebras are a family of diagram algebras. They contain many interesting subalgebras: rook algebras, Brauer algebras, Motzkin algebras, Temperley-Lieb algebras and symmetric group algebras. In this paper, we generalize the rook-Brauer algebras and their subalgebras by allowing more structured diagrams. We introduce equivariance by labelling edges of a diagram with elements of a group $G$. We introduce braiding by insisting that when two strands cross, they do so as either an under-crossing or an over-crossing. We also introduce equivariant, braided diagrams by combining these structures. We then study the homology of our diagram algebras, as pioneered by Boyd and Hepworth, using methods introduced by Boyde. We show that, given certain invertible parameters, we can identify the homology of our generalized diagram algebras with the group homology of the braid groups $B_n$ and the semi-direct products $G^n\rtimes \Sigma_n$ and $G^n\rtimes B_n$. This allows us to deduce homological stability results for our generalized diagram algebras. We also prove that for diagrams with an odd number of edges, the homology of equivariant Brauer algebras and equivariant Temperley-Lieb algebras can be identified with the group homology of $G^n\rtimes \Sigma_n$ and $G^n$ respectively, without any conditions on parameters.
\end{abstract}

\section*{Introduction}

Rook-Brauer algebras are a family of diagram algebras. They contain the Brauer algebras, Temperley-Lieb algebras, rook algebras, Motzkin algebras and symmetric groups algebras as subalgebras. These algebras have been established for many years and their representation theory is a rich area of research (see the histories of these algebras given later in this introduction). By comparison, the homology of the these algebras is a very young field indeed. The idea of studying the homology and homological stability of these (and related) algebras was pioneered by Boyd and Hepworth (\cite{BH1} and \cite{Hepworth}), who considered the homology of Temperley-Lieb algebras and the homology of Iwahori-Hecke algebras of type $A$. These papers have inspired many recent results on the homology of algebras: Iwahori-Hecke algebras of type $B$ have been studied in \cite{Moselle}; homological stability for Brauer algebras was proved in \cite{BHP}; the homology of rook algebras and rook-Brauer algebras has been developed in \cite{Boyde}; the homology of partition algebras has been studied in \cite{BHP2}; and the homology Graham-Lehrer cellular algebras has been studied in \cite{Boyde-cellular}.

In this paper we introduce generalizations of the rook-Brauer diagrams and their algebras by introducing equivariance and braiding, both separately and concurrently. Equivariance is introduced by adding labels from a group $G$ onto each of the edges in a rook-Brauer diagram. Braiding is introduced by insisting that when two edges meet they form an over-crossing or an under-crossing. With these two cases we can define the equivariant, braided rook-Brauer diagrams, where diagrams are braided and each strand comes with a label from the group $G$. Furthermore, since the rook-Brauer algebras have many interesting subalgebras, we obtain equivariant and braided versions of the Brauer algebras, Temperley-Lieb algebras, rook algebras and Motzkin algebras.

The study of the homology of algebras introduced by Boyd and Hepworth involves expressing the homology of a diagram algebra, $A$, as $\mathrm{Tor}_{\star}^A\left(\mathbbm{1},\mathbbm{1}\right)$ for some choice of trivial module $\mathbbm{1}$. In many cases we take the trivial module to be a copy of the ground ring, $k$, where a diagram acts as $1\in k$ if it has the maximum number of left-to-right connections and acts as zero otherwise.

We study the homology of our generalized rook-Brauer algebras and their subalgebras when certain parameters are invertible, by employing the methods introduced by Boyde \cite{Boyde}. Boyde studies the homology of rook algebras and rook-Brauer algebras using a theory of idempotents and ideals in the rook-Brauer algebras. He also shows that this framework can be used to recover results about Brauer algebras from \cite{BHP}. Two important results in this paper use this theory to identify the homology of rook algebras and the homology of Brauer algebras with the group homology of the symmetric group when certain parameters are invertible.

In this paper, by identifying the correct idempotents and ideals, we generalize these results to other families of groups. 

In Section \ref{gen-rook-hom-sec}, given certain invertible parameters and a group $G$, we will show that:
\begin{itemize}
\item the homology of $G$-equivariant rook algebras can be identified with the group homology of the semi-direct product $G^n\rtimes \Sigma_n$;
\item the homology of the braided rook algebras can be identified with the homology of the braid groups;
\item the homology of $G$-equivariant braided rook algebras can be identified with the group homology of the semi-direct product $G^n\rtimes B_n$;
\item we have homological stability results for these generalized rook algebras.
\end{itemize}

In Section \ref{gen-brauer-sec}, given certain invertible parameters and an abelian group $G$, we will show that:
\begin{itemize}
\item the homology of $G$-equivariant Brauer algebras can be identified with the group homology of the semi-direct product $G^n\rtimes \Sigma_n$;
\item the homology of the braided Brauer algebras can be identified with the homology of the braid groups;
\item the homology of $G$-equivariant braided Brauer algebras can be identified with the group homology of the semi-direct product $G^n\rtimes B_n$;
\item we have homological stability results for these generalized Brauer algebras;
\item the homology of the $G$-equivariant Temperley-Lieb algebra can be identified with the group homology of $G^n$.
\end{itemize}

In Section \ref{gen-RB-sec}, given certain invertible parameters, we will show that:
\begin{itemize}
\item the homology of generalized rook-Brauer algebras coincides with the homology of generalized Brauer algebras;
\item we have homological stability results for these generalized rook-Brauer algebras.
\end{itemize}

In Section \ref{motzkin-sec}, we study the homology of the Motzkin algebras. To the best of the author's knowledge the homology of the Motzkin algebras has not been studied before. Given certain invertible parameters, we will show that:
\begin{itemize}
\item in the non-equivariant case, the homology of the Motzkin algebras coincides with the homology of the Temperley-Lieb algebras;
\item we obtain a number of vanishing results and a statement on homological stability by combining our result with those of \cite{BH1};
\item in the equivariant case, given an abelian group $G$, the homology of the $G$-equivariant Motzkin algebras coincides with the homology of the $G$-equivariant Temperley-Lieb algebras. 
\end{itemize}

In Section \ref{sroka-sec} we prove some results about the homology of $G$-equivariant Brauer algebras and $G$-equivariant Temperley-Lieb algebras when the parameters are not invertible and $G$ is abelian. We show that:
\begin{itemize}
\item the homology of a $G$-equivariant Brauer algebra with an odd number of edges, $n$, coincides with the homology of $G^n\rtimes \Sigma_n$;
\item the homology of a $G$-Temperley-Lieb algebra with an odd number of edges, $n$, coincides with the group homology of $G^n$.
\end{itemize}

In order to set the scene for our generalized rook-Brauer algebras and their homology, we give a brief history of homological stability and of each of the algebras that we will be generalizing.

\subsection*{Homological stability}
Homological stability for groups refers to the following type of result. A family of groups $G_1\subset G_2 \subset G_3 \subset \cdots$ is said to satisfy homological stability if the group homology $H_i(G_n)$ is independent of $n$, up to isomorphism, for sufficiently large $n$. There are many famous examples of homological stability for groups: the symmetric groups \cite{Nakaoka}; the braid groups \cite{Arnold1}; general linear groups \cite{Charney}, \cite{VDK}; mapping class groups \cite{Harer}, \cite{HW}; automorphism groups of free groups \cite{HV}, \cite{HVW}, \cite{HV1}, \cite{HW2}; families of Coxeter groups \cite{Hepworth2}; diffeomorphism groups \cite{Tillmann}; automorphism groups \cite{RWW}; Higman-Thompson groups \cite{SW}; Artin monoids and Artin groups \cite{Boyd1}. 

Recently, the notion of homological stability for algebras has been introduced by Boyd and Hepworth. This involves a similar statement on homology, where the homology of algebras is interpreted in terms of certain $\mathrm{Tor}$ functors. This approach has been used to study Iwahori-Hecke algebras of type $A$ in \cite{Hepworth}; Temperley-Lieb algebras in \cite{BH1}; Iwahori-Hecke algebras of type $B$ in \cite{Moselle}; Brauer algebras in \cite{BHP}; rook-Brauer algebras in \cite{Boyde}; partition algebras in \cite{BHP2} and \cite{Boyde2}; and cellular algebras in \cite{Boyde-cellular}.

\subsection*{Brauer algebras}
Brauer algebras first appeared in \cite{Brauer}, in the study of the representation theory of orthogonal groups and symplectic groups. The representation theory of Brauer algebras and their generalizations has been well-studied (see \cite{CDDM}, \cite{CDM}, \cite{CDM2}, \cite{CDM3},\cite{Martin2}, \cite{DM}, \cite{KMP} for example).

The homology of the Brauer algebras has been studied in \cite{BHP} and \cite{Boyde}.

\subsection*{Temperley-Lieb algebras}
Temperley-Lieb algebras first appeared in \cite{TL} in the study of statistical mechanics. They later arose independently in work of Jones (\cite{Jones1} and \cite{Jones2}), studying von Neumann algebras and knot invariants. The representation theory and applications of Temperley-Lieb algebras and their variants have been well-studied (see \cite{MartinTL1}, \cite{MartinTL2}, \cite{MartinTL3}, \cite{MartinTL4}, \cite{MartinTL}, \cite{Westbury}, \cite{Abrahamsky}, \cite{RSA} for example).

The homology of the Temperley-Lieb algebras has been studied in \cite{BH1}, \cite{RW-pre}, \cite{Sroka} and \cite{Boyde}. 

\subsection*{Rook algebras}
The rook monoid (also known as the symmetric inverse semigroup) is the monoid of partial bijections on a set. The rook algebra is the corresponding monoid algebra. The naming convention comes from \cite{Solomon}, as the elements of the rook monoid are in bijection with the placement of non-attacking rooks on a chessboard. The representation theory of rook algebras and their generalizations has been widely studied (see \cite{HalvR}, \cite{Solomon}, \cite{Grood}, \cite{Halverson}, \cite{FHH}, \cite{BM}, \cite{Xiao}, \cite{Campbell} for example).

The homology of rook algebras has been studied in \cite{Boyde}.

\subsection*{Motzkin algebras}
Motzkin algebras were first introduced in \cite{BHal}. The name derives from the fact that the dimension of the $n^{th}$ Motzkin algebra is the $2n^{th}$ Motzkin number. These first appeared in \cite{Motzkin} and are defined as the number of ways of drawing non-intersecting chords between $2n$ points on a circle.  In \cite{BHal} it is shown that the Motzkin algebra appears as the centralizer algebra of a quantum enveloping algebra acting on a tensor power of irreducible modules. They have been further studied in \cite{JoYa}. The Motzkin algebras contain the Temperley-Lieb algebras as subalgebras.

\subsection*{Rook-Brauer algebras}
These are also known as \emph{partial Brauer algebras}. They first appeared in \cite{GW} and were further studied in \cite{KM}. Their representation theory has been studied in \cite{MM} and \cite{HD}. For example, in \cite{HD}, it is shown that the rook-Brauer algebras appear as the centralizer of an orthogonal group acting on a tensor algebra. The homology of the rook-Brauer algebras is studied in \cite{Boyde}.

Rook-Brauer algebras contain the Brauer algebras, the Temperley-Lieb algebras, rook algebras, Motzkin algebras and the symmetric group algebras.

\subsection*{Structure of the paper}
The paper is organized as follows.

In Section \ref{diagrams-sec}, we recall the definitions of rook-Brauer diagrams, rook diagrams, Brauer diagrams, Motzkin diagrams and  Temperley-Lieb diagrams. We then define our generalized versions of these diagrams: the equivariant diagrams, the braided diagrams and the equivariant braided diagrams.

In Section \ref{composition-sec}, we define how to compose generalized rook-Brauer diagrams. We discuss why we restrict to abelian groups for our results in Sections \ref{gen-brauer-sec}-\ref{sroka-sec}.

In Section \ref{generalized-algebra-sec}, we define the generalized rook-Brauer algebras and their subalgebras using the diagrams defined in Section \ref{diagrams-sec} and the composition defined in Section \ref{composition-sec}.

In Sections \ref{link-state-sec}, \ref{technical-results-sec} and \ref{group-alg-sec} we introduce the necessary generalizations of Boyde's results on ideals, idempotents and group algebras that we require to prove our results.

In Sections \ref{gen-rook-hom-sec}, \ref{gen-brauer-sec}, \ref{gen-RB-sec}, \ref{motzkin-sec}, \ref{sroka-sec}, we will prove the results outlined earlier in this introduction.

\subsection*{Acknowledgements}

My interest in the homology of diagram algebras was sparked by Guy Boyde's excellent talk at the $36\textsuperscript{th}$ British Topology Meeting at the University of Sheffield. I would like to thank Guy for his talk and the organizers (Brad Ashley, James Brotherston, James Cranch, Andrew Fisher, Paul Mitchener, Neil Strickland, Markus Szymik, Leyna Watson May and Sarah Whitehouse) for putting on such an enjoyable conference. I would like to thank Rachael Boyd and Guy Boyde for taking the time to read earlier versions of this paper and for all their helpful comments and suggestions. I would like to thank Leyna Watson May, Sarah Whitehouse, James Brotherston, Andrew Fisher, Andrew Neate, Jake Saunders, Jack Davidson and Guy Boyde for helpful and interesting conversations during the writing of this paper.  Finally, I would like to apologize to Natasha Cowley and James Brotherston, who probably heard more about this project whilst we were on holiday in Prague than they would have liked.

\section{Diagrams}
\label{diagrams-sec}
In this section we recall the definition of rook-Brauer diagrams, together with the subsets of rook diagrams, Brauer diagrams, Motzkin diagrams and Temperley-Lieb diagrams. We will then define our generalizations of these diagrams, where we introduce equivariance and braiding. We note that our notion of a generalized Temperley-Lieb algebra is not the same as that occurring in \cite{Green}.   

\subsection{Rook-Brauer diagrams}
\label{RB-diag-subsec}
We begin this section by recalling the definitions of rook-Brauer diagrams, rook diagrams, Brauer diagrams, Motzkin diagrams and  Temperley-Lieb diagrams.

\begin{defn}
We define diagrams as follows:
\begin{enumerate}
\item A \emph{rook-Brauer $n$-diagram} is a graph consisting of two columns of $n$ vertices such that each vertex is connected to at most one other by an edge. 
\item A \emph{rook $n$-diagram} is a rook-Brauer $n$-diagram having no left-to-left connections and no right-to-right connections.
\item A \emph{Brauer $n$-diagram} is a rook-Brauer $n$-diagram having no missing edges.
\item A \emph{Motzkin $n$-diagram} is a rook-Brauer $n$-diagram that is planar.
\item A \emph{Temperley-Lieb $n$-diagram} is a rook-Brauer $n$-diagram that is planar and has no missing edges.
\end{enumerate}
\end{defn}

\begin{eg}
Fix $n=4$. Consider the following diagrams:
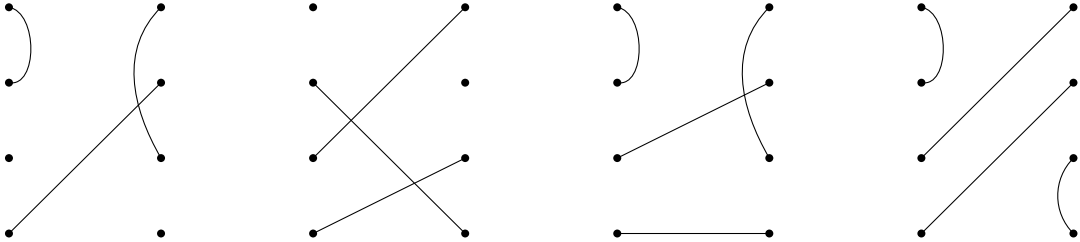
\begin{center}
\begin{tikzpicture}
\node[circle,fill=black,inner sep=0pt,minimum size=3pt] (a) at (0,0) {};
\node[circle,fill=black,inner sep=0pt,minimum size=3pt] (a) at (0,1) {};
\node[circle,fill=black,inner sep=0pt,minimum size=3pt] (a) at (0,2) {};
\node[circle,fill=black,inner sep=0pt,minimum size=3pt] (a) at (0,3) {};
\node[circle,fill=black,inner sep=0pt,minimum size=3pt] (a) at (2,0) {};
\node[circle,fill=black,inner sep=0pt,minimum size=3pt] (a) at (2,1) {};
\node[circle,fill=black,inner sep=0pt,minimum size=3pt] (a) at (2,2) {};
\node[circle,fill=black,inner sep=0pt,minimum size=3pt] (a) at (2,3) {};
\node[circle,fill=black,inner sep=0pt,minimum size=3pt] (a) at (4,0) {};
\node[circle,fill=black,inner sep=0pt,minimum size=3pt] (a) at (4,1) {};
\node[circle,fill=black,inner sep=0pt,minimum size=3pt] (a) at (4,2) {};
\node[circle,fill=black,inner sep=0pt,minimum size=3pt] (a) at (4,3) {};
\node[circle,fill=black,inner sep=0pt,minimum size=3pt] (a) at (6,0) {};
\node[circle,fill=black,inner sep=0pt,minimum size=3pt] (a) at (6,1) {};
\node[circle,fill=black,inner sep=0pt,minimum size=3pt] (a) at (6,2) {};
\node[circle,fill=black,inner sep=0pt,minimum size=3pt] (a) at (6,3) {};
\node[circle,fill=black,inner sep=0pt,minimum size=3pt] (a) at (8,0) {};
\node[circle,fill=black,inner sep=0pt,minimum size=3pt] (a) at (8,1) {};
\node[circle,fill=black,inner sep=0pt,minimum size=3pt] (a) at (8,2) {};
\node[circle,fill=black,inner sep=0pt,minimum size=3pt] (a) at (8,3) {};
\node[circle,fill=black,inner sep=0pt,minimum size=3pt] (a) at (10,0) {};
\node[circle,fill=black,inner sep=0pt,minimum size=3pt] (a) at (10,1) {};
\node[circle,fill=black,inner sep=0pt,minimum size=3pt] (a) at (10,2) {};
\node[circle,fill=black,inner sep=0pt,minimum size=3pt] (a) at (10,3) {};
\node[circle,fill=black,inner sep=0pt,minimum size=3pt] (a) at (12,0) {};
\node[circle,fill=black,inner sep=0pt,minimum size=3pt] (a) at (12,1) {};
\node[circle,fill=black,inner sep=0pt,minimum size=3pt] (a) at (12,2) {};
\node[circle,fill=black,inner sep=0pt,minimum size=3pt] (a) at (12,3) {};
\node[circle,fill=black,inner sep=0pt,minimum size=3pt] (a) at (14,0) {};
\node[circle,fill=black,inner sep=0pt,minimum size=3pt] (a) at (14,1) {};
\node[circle,fill=black,inner sep=0pt,minimum size=3pt] (a) at (14,2) {};
\node[circle,fill=black,inner sep=0pt,minimum size=3pt] (a) at (14,3) {};
\draw (0,2) to[out=-10,in=-10] (0,3);
\draw (2,1) to[out=120,in=225] (2,3);
\draw (0,0) -- (2,2);
\draw (4,0) -- (6,1);
\draw (4,1) -- (6,3);
\draw (4,2) -- (6,0);
\draw (8,0) -- (10,0);
\draw (8,1) -- (10,2);
\draw (8,2) to[out=-10,in=-10] (8,3);
\draw (10,1) to[out=120,in=225] (10,3);
\draw (12,0) -- (14,2);
\draw (12,1) -- (14,3);
\draw (12,2) to[out=-10,in=-10] (12,3);
\draw (14,0) to[out=135,in=225] (14,1);
\end{tikzpicture}
\end{center}
\begin{itemize}
\item The first (leftmost) diagram is a rook-Brauer $4$-diagram. We have four nodes in each column and each node is connected to at most one other by an edge. We note that it cannot be a rook $4$-diagram since it has left-to-left and right-to-right connections. It also cannot be a Brauer $4$-diagram since it has a missing edge. Therefore, it also cannot be a Temperley-Lieb $4$-diagram (it's also not planar).
\item The second diagram is a rook $4$-diagram, since we have no left-to-left and no right-to-right connections. It cannot be a Brauer diagram since we have a missing edge. It cannot be a Motzkin diagram since it is not planar.
\item The third diagram is a Brauer $4$-diagram since there are no missing edges. It cannot be a rook $4$-diagram since we have left-to-left and right-to-right connections. It also cannot be a Motzkin $4$-diagram or a Temperley-Lieb $4$-diagram since it is not planar.
\item The fourth diagram is a Temperley-Lieb $4$-diagram (and therefore a Motzkin $4$-diagram). It has no missing edges and it is planar.
\end{itemize}
\end{eg}

\subsection{Generalized rook-Brauer diagrams 1: equivariance}
\label{equiv-diag-subsec}
In this subsection we define our first generalization of the diagrams in Subsection \ref{RB-diag-subsec} by introducing equivariance.

\begin{defn}
\label{GRB-diag-defn}
Let $G$ be a group. A \emph{$G$-rook-Brauer $n$-diagram} is a graph consisting of two columns of $n$ nodes such that
\begin{itemize}
\item each node is connected to at most one other node by an edge;
\item each edge is labelled by an element of the group $G$.
\end{itemize}
\end{defn}

\begin{defn}
\label{G-diag-defn}
Let $G$ be a group. 
\begin{enumerate}
\item A \emph{$G$-rook $n$-diagram} is a $G$-rook-Brauer $n$-diagram having no left-to-left and no right-to-right connections.
\item A \emph{$G$-Brauer $n$-diagram} is a $G$-rook-Brauer $n$-diagram having no missing edges.
\item A \emph{$G$-Motzkin $n$-diagram} is a $G$-rook-Brauer diagram that is planar.
\item A \emph{$G$-Temperley-Lieb $n$-diagram} is a $G$-rook-Brauer $n$-diagram that is planar and has no missing edges.
\end{enumerate}
\end{defn}

\begin{rem}
We note that given a $G$-labelled diagram from Definition \ref{GRB-diag-defn} or Definition \ref{G-diag-defn} we can recover the diagrams of Subsection \ref{RB-diag-subsec} simply by omitting the labels.
\end{rem}

\subsection{Generalized rook-Brauer diagrams 2: braiding}
\label{braid-diag-subsec}
In this subsection we define our second generalization of the diagrams in Subsection \ref{RB-diag-subsec} by introducing braiding.

\begin{defn}
\label{BRB-diag-defn}
A \emph{braided rook-Brauer $n$-diagram} is a graph consisting of two columns of $n$ nodes such that each node is connected to at most one other node by a strand. These strands are not planar. Whenever a strand crosses another it is either an over-crossing or an under-crossing.
\end{defn}

\begin{defn}
\label{braid-diag-defn} 
We have the following subsets of diagrams.
\begin{enumerate}
\item A \emph{braided rook $n$-diagram} is a braided rook-Brauer $n$-diagram having no left-to-left and no right-to-right connections.
\item A \emph{braided Brauer $n$-diagram} is a braided rook-Brauer $n$-diagram having no missing edges.
\end{enumerate}
\end{defn}

\begin{rem}
We note that it does not make sense to talk about braided Motzkin diagrams or braided Temperley-Lieb diagrams in this setting as these diagrams have no crossings, so there is no extra data.
\end{rem}

\begin{rem}
We note that given a braided diagram from Definition \ref{BRB-diag-defn} or Definition \ref{braid-diag-defn} we can recover the rook-Brauer diagrams, rook diagrams and Brauer diagrams from Subsection \ref{RB-diag-subsec} by projecting onto the plane.
\end{rem}

\subsection{Generalized rook-Brauer diagrams 3: braiding and equivariance}
\label{equiv-braid-diag-subsec}
We can combine the extra structure from Subsections \ref{equiv-diag-subsec} and \ref{braid-diag-subsec}.

\begin{defn}
\label{GBRB-diag-defn}
Let $G$ be a group. A \emph{$G$-braided rook-Brauer $n$-diagram} is a graph consisting of two columns of $n$ nodes such that
\begin{itemize}
\item each node is connected to at most one other node by a strand (as above, whenever a strand crosses another it is either an over-crossing or an under-crossing);
\item each strand is labelled by an element of the group $G$.
\end{itemize}
\end{defn}

\begin{defn}
\label{GB-diag-defn}
Let $G$ be a group. 
\begin{enumerate}
\item A \emph{$G$-braided rook $n$-diagram} is a $G$-braided rook-Brauer $n$-diagram having no left-to-left and no right-to-right connections.
\item A \emph{$G$-braided Brauer $n$-diagram} is a $G$-braided rook-Brauer $n$-diagram having no missing edges.
\end{enumerate}
\end{defn}

\begin{rem}
Suppose we are given a $G$-braided diagram as in Definition \ref{GBRB-diag-defn} or Definition \ref{GB-diag-defn}. We could project the diagram onto the plane, in which case we would recover the equivariant diagrams of Subsection \ref{equiv-diag-subsec}. On the other hand we could omit the labels, in which case we would recover the braided diagrams of Subsection \ref{braid-diag-subsec}.
\end{rem}

\subsection{Extensions of diagrams}
Given any sort of $n$-diagram from Subsections \ref{RB-diag-subsec}, \ref{equiv-diag-subsec}, \ref{braid-diag-subsec} and \ref{equiv-braid-diag-subsec}, we can extend it to an $(n+1)$-diagram by adding an extra vertex at the end of each column and joining them with a horizontal edge. In the equivariant cases, we take the label on our extra horizontal edge to be $1\in G$.

\begin{eg}
Consider the diagrams below.
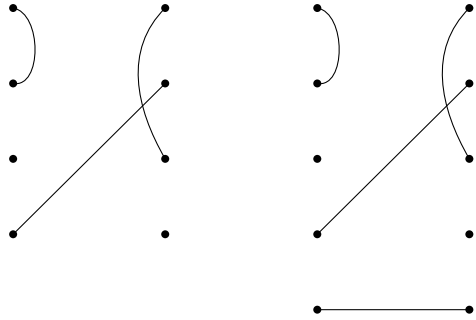
\begin{center}
\begin{tikzpicture}
\node[circle,fill=black,inner sep=0pt,minimum size=3pt] (a) at (0,0) {};
\node[circle,fill=black,inner sep=0pt,minimum size=3pt] (a) at (0,1) {};
\node[circle,fill=black,inner sep=0pt,minimum size=3pt] (a) at (0,2) {};
\node[circle,fill=black,inner sep=0pt,minimum size=3pt] (a) at (0,3) {};
\node[circle,fill=black,inner sep=0pt,minimum size=3pt] (a) at (2,0) {};
\node[circle,fill=black,inner sep=0pt,minimum size=3pt] (a) at (2,1) {};
\node[circle,fill=black,inner sep=0pt,minimum size=3pt] (a) at (2,2) {};
\node[circle,fill=black,inner sep=0pt,minimum size=3pt] (a) at (2,3) {};
\draw (0,2) to[out=-10,in=-10] (0,3);
\draw (2,1) to[out=120,in=225] (2,3);
\draw (0,0) -- (2,2);
\node[circle,fill=black,inner sep=0pt,minimum size=3pt] (a) at (4,0) {};
\node[circle,fill=black,inner sep=0pt,minimum size=3pt] (a) at (4,1) {};
\node[circle,fill=black,inner sep=0pt,minimum size=3pt] (a) at (4,2) {};
\node[circle,fill=black,inner sep=0pt,minimum size=3pt] (a) at (4,3) {};
\node[circle,fill=black,inner sep=0pt,minimum size=3pt] (a) at (6,0) {};
\node[circle,fill=black,inner sep=0pt,minimum size=3pt] (a) at (6,1) {};
\node[circle,fill=black,inner sep=0pt,minimum size=3pt] (a) at (6,2) {};
\node[circle,fill=black,inner sep=0pt,minimum size=3pt] (a) at (6,3) {};
\draw (4,2) to[out=-10,in=-10] (4,3);
\draw (6,1) to[out=120,in=225] (6,3);
\draw (4,0) -- (6,2);
\node[circle,fill=black,inner sep=0pt,minimum size=3pt] (a) at (4,-1) {};
\node[circle,fill=black,inner sep=0pt,minimum size=3pt] (a) at (6,-1) {};
\draw (4,-1) -- (6,-1);
\end{tikzpicture}
\end{center}
On the left we have a rook-Brauer $4$-diagram and on the the right we have its extension to a rook-Brauer $5$-diagram. 
\end{eg}

\section{Composition of diagrams}
\label{composition-sec}
In this section we define the composition of diagrams without the presence of a ring. These compositions will underpin the multiplication in our generalized algebras in Section \ref{generalized-algebra-sec}. We choose to have a short section defining these compositions (rather than including them directly into the definitions of the algebras in Section \ref{generalized-algebra-sec}) for two reasons. The first is to do with restrictions on the group for equivariant versions of diagram algebras. For those algebras in which we can form loops by composition we must restrict to abelian groups. However, for rook algebras we may work with any group. We believe this is best explained by considering the composition of diagrams without the rest of the structure that we will have in our generalized diagram algebras. The second reason is one of convenience. We can define all of our generalized algebras and subalgebras using just three definitions of compositions.  

\subsection{Composition of rook-Brauer diagrams}
We start by recalling the definition of composition for rook-Brauer diagrams.
\begin{defn}
\label{rook-Brauer-comp-defn}
Given two  rook-Brauer $n$-diagrams, $x$ and $y$, the product $x\cdot y$ is the diagram obtained as follows:
\begin{enumerate}
\item concatenate $x$ and $y$ by identifying the right-hand nodes of $x$ with the left-hand nodes of $y$;
\item forget the vertices in the middle;
\item if we form a loop we delete it.
\end{enumerate}
By restriction, this also defines composition for rook diagrams, Brauer diagrams and Temperley-Lieb diagrams.
\end{defn}

\subsection{Composition of $G$-rook diagrams and $G$-braided rook diagrams}

\begin{defn}
\label{G-rook-alg-comp-defn}
Let $G$ be a group. Given two $G$-rook $n$-diagrams (resp. $G$-braided rook $n$-diagrams),  $x$ and $y$, the product $x\cdot y$ is the diagram obtained as follows:
\begin{enumerate}
\item concatenate $x$ and $y$ by identifying the right-hand nodes of $x$ with the left-hand nodes of $y$;
\item forget the vertices in the middle;
\item the label on an edge formed by concatenation is the product of the group elements, read from left to right.
\end{enumerate}
\end{defn}

\begin{eg}
The following diagram gives an example of composition of $G$-rook $3$-diagrams.
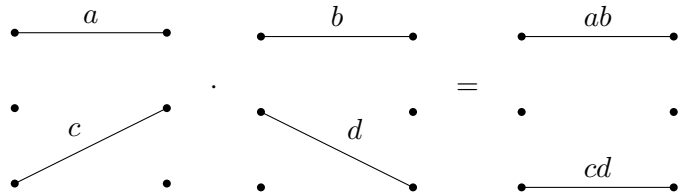
\begin{center}
\begin{tikzpicture}[
    tlabel/.style={pos=0.4,right=-1pt},
    baseline=(current bounding box.center)
    ]
\node[circle,fill=black,inner sep=0pt,minimum size=3pt] (a) at (0,0) {};
\node[circle,fill=black,inner sep=0pt,minimum size=3pt] (a) at (0,1) {};
\node[circle,fill=black,inner sep=0pt,minimum size=3pt] (a) at (0,2) {};
\node[circle,fill=black,inner sep=0pt,minimum size=3pt] (a) at (2,0) {};
\node[circle,fill=black,inner sep=0pt,minimum size=3pt] (a) at (2,1) {};
\node[circle,fill=black,inner sep=0pt,minimum size=3pt] (a) at (2,2) {};
\draw (0,2) edge["$a$"] (2,2);
\draw (0,0) edge["$c$"] (2,1);
\end{tikzpicture}
$\quad \cdot \quad$
\begin{tikzpicture}[
    tlabel/.style={pos=0.4,right=-1pt},
    baseline=(current bounding box.center)
    ]
\node[circle,fill=black,inner sep=0pt,minimum size=3pt] (a) at (4,0) {};
\node[circle,fill=black,inner sep=0pt,minimum size=3pt] (a) at (4,1) {};
\node[circle,fill=black,inner sep=0pt,minimum size=3pt] (a) at (4,2) {};
\node[circle,fill=black,inner sep=0pt,minimum size=3pt] (a) at (6,0) {};
\node[circle,fill=black,inner sep=0pt,minimum size=3pt] (a) at (6,1) {};
\node[circle,fill=black,inner sep=0pt,minimum size=3pt] (a) at (6,2) {};
\draw (4,2) edge["$b$"] (6,2);
\draw (4,1) edge["$d$"] (6,0);
\end{tikzpicture}
$\quad = \quad$
\begin{tikzpicture}[
    tlabel/.style={pos=0.4,right=-1pt},
    baseline=(current bounding box.center)
    ]
\node[circle,fill=black,inner sep=0pt,minimum size=3pt] (a) at (8,0) {};
\node[circle,fill=black,inner sep=0pt,minimum size=3pt] (a) at (8,1) {};
\node[circle,fill=black,inner sep=0pt,minimum size=3pt] (a) at (8,2) {};
\node[circle,fill=black,inner sep=0pt,minimum size=3pt] (a) at (10,0) {};
\node[circle,fill=black,inner sep=0pt,minimum size=3pt] (a) at (10,1) {};
\node[circle,fill=black,inner sep=0pt,minimum size=3pt] (a) at (10,2) {};
\draw (8,2) edge["$ab$"] (10,2);
\draw (8,0) edge["$cd$"] (10,0);
\end{tikzpicture}
\end{center}
\end{eg}

\subsection{Composition of generalized rook-Brauer diagrams}

\begin{defn}
\label{G-rook-brauer-comp-defn}
Let $G$ be an abelian group. Given two $G$-rook-Brauer $n$-diagrams (resp. braided rook-Brauer $n$-diagrams or $G$-braided rook-Brauer diagrams), $x$ and $y$, the product $x\cdot y$ is the diagram obtained as follows:
\begin{enumerate}
\item concatenate $x$ and $y$ by identifying the right-hand nodes of $x$ with the left-hand nodes of $y$;
\item forget the vertices in the middle;
\item if we have $G$-labels, the label on an edge formed by concatenation is the product of the group elements;
\item if we form a loop we delete it, irrespective of crossings or $G$-labels.
\end{enumerate}
By restriction, this also defines composition for braided rook algebras, generalized Brauer diagrams and generalized Temperley-Lieb diagrams.
\end{defn}

\begin{rem}
\label{abelian-rem}
The composition of $G$-rook diagrams and $G$-braided rook diagrams in Definition \ref{G-rook-alg-comp-defn} holds for any group $G$. However, when defining composition for $G$-rook-Brauer diagrams and $G$-braided rook-Brauer diagrams we require $G$ to be abelian. This is due to the fact that we can form loops when composing these diagrams. For example, consider the following three $G$-rook-Brauer $4$-diagrams. Note that, for simplicity, we will only put labels on the necessary edges.

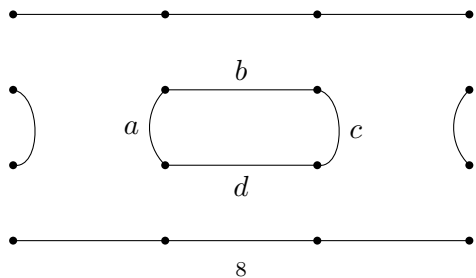
\begin{center}
\begin{tikzpicture}
\node[circle,fill=black,inner sep=0pt,minimum size=3pt] (a) at (0,0) {};
\node[circle,fill=black,inner sep=0pt,minimum size=3pt] (a) at (0,1) {};
\node[circle,fill=black,inner sep=0pt,minimum size=3pt] (a) at (0,2) {};
\node[circle,fill=black,inner sep=0pt,minimum size=3pt] (a) at (0,3) {};
\node[circle,fill=black,inner sep=0pt,minimum size=3pt] (a) at (2,0) {};
\node[circle,fill=black,inner sep=0pt,minimum size=3pt] (a) at (2,1) {};
\node[circle,fill=black,inner sep=0pt,minimum size=3pt] (a) at (2,2) {};
\node[circle,fill=black,inner sep=0pt,minimum size=3pt] (a) at (2,3) {};
\node[circle,fill=black,inner sep=0pt,minimum size=3pt] (a) at (4,0) {};
\node[circle,fill=black,inner sep=0pt,minimum size=3pt] (a) at (4,1) {};
\node[circle,fill=black,inner sep=0pt,minimum size=3pt] (a) at (4,2) {};
\node[circle,fill=black,inner sep=0pt,minimum size=3pt] (a) at (4,3) {};
\node[circle,fill=black,inner sep=0pt,minimum size=3pt] (a) at (6,0) {};
\node[circle,fill=black,inner sep=0pt,minimum size=3pt] (a) at (6,1) {};
\node[circle,fill=black,inner sep=0pt,minimum size=3pt] (a) at (6,2) {};
\node[circle,fill=black,inner sep=0pt,minimum size=3pt] (a) at (6,3) {};
\draw (0,3) -- (2,3) -- (4,3) -- (6,3);
\draw (0,0) -- (2,0) -- (4,0) -- (6,0);
\draw (2,2) edge["$b$"] (4,2);
\draw (2,1) edge["$d$", swap] (4,1);
\draw (4,1) to[out=-10,in=-10, "$c$",swap] (4,2);
\draw (2,1) to[out=135,in=225, "$a$"] (2,2);
\draw (0,1) to[out=-10,in=-10] (0,2);
\draw (6,1) to[out=135,in=225] (6,2);
\end{tikzpicture}
\end{center} 
We want to be able to compose these diagrams and we want the composition to be associative. If we perform composition (in either order) in the diagram above, we will have a loop in the middle. When we come to study the algebras formed by these diagrams it will be important to keep track of the labels on this loop. If $G$ is abelian, we can simply take the product of the labels. If $G$ is not abelian then there is not a straightforward choice of labels on the composite of the left-hand two diagrams \emph{and} on the composite of the right-hand two diagrams that leads to an associative composition. 
\end{rem}

\section{Generalized rook-Brauer algebras}
\label{generalized-algebra-sec}

We begin by recalling the definition of the rook-Brauer algebras and their subalgebras. We then proceed to define the generalized rook-Brauer algebras and their subalgebras. For the two cases where we are building in equivariance, we begin by treating the generalized rook algebra case first since this is defined for any group $G$. On the other hand, when defining the generalized rook-Brauer algebras, Brauer algebras, Motzkin algebras and Temperley-Lieb algebras we wish to restrict to the case where $G$ is abelian. Loosely speaking, this is because we cannot form loops by composing rook diagrams but we can when composing our other types of diagram as discussed in Remark \ref{abelian-rem}. 

We begin by recalling a definition that we will require throughout this section and throughout the rest of the paper: the definition of a ring with an action of a group.

\begin{defn}
Let $G$ be a group. A \emph{commutative $G$-ring} is a unital, commutative ring $k$ together with a group homomorphism $\rho\colon G\rightarrow \mathrm{Aut}(k)$, where $\mathrm{Aut}(k)$ denotes the group of ring automorphisms of $k$. In what follows we will suppress the $\rho$ and write $g\lambda$ for $\rho(g)(\lambda)$.
\end{defn}

\subsection{Rook-Brauer algebras}
\label{RB-subsec}

We begin by recalling the definitions of rook-Brauer algebras, rook algebras, Brauer algebras and Temperley-Lieb algebras.

\begin{defn}
Let $k$ be a commutative ring. Let $n$ be a natural number. Let $\delta, \varepsilon \in k$. The \emph{rook-Brauer algebra}, written $\mathcal{RB}_n(\delta, \varepsilon)$ consists of formal $k$-linear combinations of rook-Brauer $n$-diagrams with the bilinear extension of the composition of Definition \ref{rook-Brauer-comp-defn} with the conditions:
\begin{enumerate}
\item when we remove a loop from the composite we multiply by a pre-factor of $\delta$;
\item replace each contractible component with a pre-factor of $\varepsilon$.
\end{enumerate}
\end{defn}

\begin{defn}
By restriction, we obtain four subalgebras: the \emph{rook algebra}, the \emph{Brauer algebra}, the \emph{Motzkin algebra} and the \emph{Temperley-Lieb algebra}. We denote these by $\mathcal{R}_n(\varepsilon)$, $\mathcal{B}_n(\delta)$, $\mathcal{M}_n\left(\delta, \varepsilon\right)$ and $\mathcal{TL}_n(\delta)$ respectively.
\end{defn}

\begin{rem}
As explained in \cite[Subsection 1.2]{Boyde}, we note that the composite of two rook diagrams can yield no loops since a rook diagram has no left-to-left connections and no right-to-right connections. Therefore the rook algebra is independent of $\delta$ and we drop it from the notation. Similarly, the composite of Brauer diagrams or Temperley-Lieb diagrams can have no contractible components since the diagrams have no missing edges. Therefore these algebras are independent of $\varepsilon$ and we drop it from the notation.
\end{rem}

\subsection{Equivariant rook algebras and equivariant braided rook algebras}
\label{ER-subsec}

\begin{defn}
\label{GR-defn}
Let $G$ be a group. Let $k$ be a commutative $G$-ring. Let $n$ be a natural number. Let $ \varepsilon \in k$. 
\begin{enumerate}
\item The \emph{$G$-rook algebra}, written $G\mathcal{R}_n(\varepsilon)$, consists of formal $k$-linear combinations of $G$-rook $n$-diagrams with the bilinear extension of the composition from Definition \ref{G-rook-alg-comp-defn} with the condition that we replace each contractible component with a pre-factor of $\varepsilon$.
\item The \emph{$G$-braided rook algebra}, written $GB\mathcal{R}_n(\varepsilon)$, consists of formal $k$-linear combinations of $G$-braided rook $n$-diagrams with the bilinear extension of the composition from Definition \ref{G-rook-alg-comp-defn} with the condition that we replace each contractible component with a pre-factor of $\varepsilon$.
\end{enumerate}
\end{defn}

\subsection{Equivariant rook-Brauer algebras}
\label{ERB-subsec}

\begin{defn}
\label{GRB-defn}
Let $G$ be an abelian group. Let $k$ be a commutative $G$-ring. Let $n$ be a natural number. Let $\delta, \varepsilon \in k$. The \emph{$G$-rook-Brauer algebra}, written $G\mathcal{RB}_n(\delta, \varepsilon)$, consists of formal $k$-linear combinations of  $G$-rook-Brauer $n$-diagrams with the bilinear extension of the composition from Definition \ref{G-rook-brauer-comp-defn} with the conditions:
\begin{enumerate}
\item if we form a loop with left-hand label $g$ and right-hand label $h$, we replace this loop with a pre-factor of $(gh)\delta$;
\item replace each contractible component with a pre-factor of $\varepsilon$.
\end{enumerate}
\end{defn}

\begin{defn}
By restriction, we obtain three subalgebras: the \emph{$G$-Brauer algebra}, the \emph{$G$-Motzkin algebra} and the \emph{$G$-Temperley-Lieb algebra}. We denote these by $G\mathcal{B}_n(\delta)$, $G\mathcal{M}_n\left(\delta , \varepsilon\right)$ and $G\mathcal{TL}_n(\delta)$ respectively.
\end{defn}

\subsection{Braided rook-Brauer algebras}
\label{BRB-subsec}

\begin{defn}
\label{BRB-defn}
Let $k$ be a unital commutative ring. Let $n$ be a natural number. Let $\delta, \varepsilon \in k$. The \emph{braided rook-Brauer algebra}, written $B\mathcal{RB}_n(\delta, \varepsilon)$, consists of formal $k$-linear combinations of braided rook-Brauer $n$-diagrams with the bilinear extension of the composition in Definition \ref{G-rook-brauer-comp-defn} with the following conditions:
\begin{enumerate}
\item replace any loops (irrespective of crossings) with a pre-factor of $\delta$;
\item replace each contractible component with a pre-factor of $\varepsilon$.
\end{enumerate}
\end{defn}

\begin{defn}
By restriction, we obtain two subalgebras: the \emph{braided rook algebra} and the \emph{braided Brauer algebra}. We denote these by $B\mathcal{R}_n(\varepsilon)$ and $B\mathcal{B}_n(\delta)$ respectively.
\end{defn}

\begin{rem}
Braided analogues of Brauer algebras exist in the literature in the form of Birman–Murakami–Wenzl algebras (also known as BMW algebras or Birman-Wenzl algebras, see \cite{M} and \cite{BW}) and Kauffman tangle algebras (see \cite{MT}). Our braided Brauer algebras differ from these algebras as we do not assume that the \emph{Kauffman skein relation} is satisfied.
\end{rem}

\subsection{Equivariant braided rook-Brauer algebras}
\label{EBRB-subsec}
We can combine the ideas from the previous two subsections to define equivariant braided rook-Brauer algebras.

\begin{defn}
\label{GBRB-defn}
Let $G$ be an abelian group. Let $k$ be a commutative $G$-ring. Let $n$ be a natural number. Let $\delta, \varepsilon \in k$. The \emph{$G$-braided rook-Brauer algebra}, written $GB\mathcal{RB}_n(\delta, \varepsilon)$ consists of formal $k$-linear combinations of  $G$-braided rook-Brauer $n$-diagrams with the bilinear extension of the composition of Definition \ref{G-rook-brauer-comp-defn} with the following conditions:
\begin{enumerate}
\item if we form a loop (irrespective of crossings) with left-hand label $g$ and right-hand label $h$, we replace this loop with a pre-factor of $(gh)\delta$;
\item replace each contractible component with a pre-factor of $\varepsilon$.
\end{enumerate}
\end{defn}

\begin{defn}
By restriction, we obtain a subalgebra: the \emph{$G$-braided Brauer algebra}. We denote this by $GB\mathcal{B}_n(\delta)$.
\end{defn}

\subsection{Trivial modules for generalized rook-Brauer algebras}

\begin{defn}
The \emph{trivial module}, $\mathbbm{1}$, for any algebra considered in Subsections \ref{RB-subsec}, \ref{ER-subsec}, \ref{ERB-subsec}, \ref{BRB-subsec} and \ref{EBRB-subsec} is a single copy of the ground ring $k$, where:
\begin{itemize}
\item diagrams with no missing edges, no left-to-left connections and no right-to-right connections act as multiplication by $1\in k$ and
\item all other diagrams act as $0\in k$.
\end{itemize}
\end{defn}

\section{Link states}
\label{link-state-sec}
In this section we define link states for generalized rook-Brauer algebras. Link states were first introduced by Ridout and Saint-Aubin \cite{RSA}. They in turn point out that the notion is related to \emph{parenthesis structures} in \cite{Kauffmann1}, \emph{arch configurations} in \cite{DGG} and cellular structure in \cite{GL-cellular}. Our notation will closely follow that of \cite{Boyde}.

\subsection{Link states for equivariant rook-Brauer diagrams}

\begin{defn}
Let $G$ be an abelian group. By slicing vertically down the middle of a $G$-rook-Brauer $n$-diagram, we obtain the \emph{left and right link states} of the diagram. A link state consists of $n$ nodes such that:
\begin{itemize}
\item each node may be connected to at most one other node by an edge;
\item a node not connected to any other may have a hanging edge, called a \emph{defect};
\item if we slice a left-to-right connection with label $g$ say, the defects in both the left and right link states inherit the label.
\end{itemize}
\end{defn}

\begin{defn}
Fix a natural number $n$. Let $GP_i$ be the set of link states of $G$-rook-Brauer $n$-diagrams with precisely $i$ defects.
\end{defn}

\begin{defn}
Given a link state in $GP_i$ we can perform the following three operations:
\begin{itemize}
\item we can connect two defects by adding a $G$-labelled edge between the ends of the hanging edges. We call this operation a \emph{$G$-splice};
\item we can delete a defect, leaving a missing edge. This is called a \emph{deletion};
\item we can multiply a label on a hanging edge by an element of $G$. This is called a \emph{label replacement}.
\end{itemize}
\end{defn}

\subsection{Link states for braided rook-Brauer diagrams}

\begin{defn}
Consider a braided rook-Brauer $n$-diagram, we obtain the \emph{left and right link states} of the diagram as follows:
\begin{itemize}
\item we split any left-to-right connections;
\item we split a left-to-left connection if and only if it is braided through a right-to-right connection;
\item we split a right-to-right connection if and only if it is braided through a left-to-left connection.
\end{itemize}
Therefore,
\begin{itemize}
\item each node may be connected to at most one other node by an edge and
\item a node not connected to any other may have a hanging edge, called a \emph{defect}.
\end{itemize}
\end{defn}

\begin{defn}
Fix a natural number $n$. Let $BP_i$ be the set of link states of braided rook-Brauer $n$-diagrams with precisely $i$ defects.
\end{defn}

\begin{defn}
Given a link state in $BP_i$ we can perform the following two operations:
\begin{itemize}
\item we can connect two defects by connecting the hanging edges. We call this operation a \emph{braided splice}.
\item we can delete a defect, leaving a missing edge. This is called a \emph{deletion}.
\end{itemize}
\end{defn}

\begin{rem}
When performing a splice for a braided diagram, we may have to braid hanging edges over or under other edges in the link state. We have a different splice for each choice of under-crossing or over-crossing we make.
\end{rem}

\subsection{Link states for $G$-braided rook-Brauer diagrams}

\begin{defn}
Let $G$ be an abelian group. Consider a $G$-braided rook-Brauer $n$-diagram. We obtain the \emph{left and right link states} of the diagram as follows:
\begin{itemize}
\item we split any left-to-right connections;
\item we split a left-to-left connection if and only if it is braided through a right-to-right connection;
\item we split a right-to-right connection if and only if it is braided through a left-to-left connection;
\item if we split a strand with label $g$ say, the hanging edges in both the left and right link states inherit the label.
\end{itemize}
Therefore,
\begin{itemize}
\item each node may be connected to at most one other node by an edge;
\item a node not connected to any other may have a hanging edge, called a \emph{defect}.
\end{itemize}
\end{defn}

\begin{defn}
Fix a natural number $n$. Let $GBP_i$ be the set of link states of $G$-braided rook-Brauer $n$-diagrams with precisely $i$ defects.
\end{defn}

\begin{defn}
Given a link state in $GBP_i$ we can perform the following three operations:
\begin{itemize}
\item we can connect two defects by adding a $G$-labelled strand between the ends of the hanging edges. We call this operation a \emph{braided $G$-splice}.
\item we can delete a defect, leaving a missing edge. This is called a \emph{deletion};
\item we can multiply a label on a hanging edge by an element of $G$. This is called a \emph{label replacement}.
\end{itemize}
\end{defn}

\subsection{Ideals for diagrams and link states}
\label{ideals-subsec}

\begin{defn}
Let $0\leqslant i \leqslant n$. The $k$-span of $n$-diagrams in a generalized rook-Brauer algebra having at most $i$ left-to-right connections is a two-sided ideal. We will denote this ideal by $I_i$. By convention, we will set $I_{-1}=0$.
\end{defn}

\begin{defn}
Let $G$ be an abelian group. We define $k$-modules relating to link states of generalized rook-Brauer algebras as follows.
\begin{itemize}
\item Consider a link state $p\in GP_i$. Let $GJ_p$ be the $k$-submodule of $G\mathcal{RB}_n\left(\delta , \varepsilon\right)$ with basis given by the diagrams having right link state obtained from $p$ by a (possibly empty) sequence of $G$-splices, deletions and label replacements.
\item Consider a link state $p \in BP_i$. Let $BJ_p$ be the $k$-submodule of $B\mathcal{RB}_n\left(\delta , \varepsilon\right)$ with basis given by the diagrams having right link state obtained from $p$ by a (possibly empty) sequence of braided splices and deletions.
\item Consider a link state $p \in GBP_i$. Let $GBJ_p$ be the $k$-submodule of $GB\mathcal{RB}_n\left(\delta , \varepsilon\right)$ with basis given by the diagrams having right link state obtained from $p$ by a (possibly empty) sequence of braided $G$-splices, deletions and label replacements.
\end{itemize}
\end{defn}

\begin{lem}
Let $G$ be an abelian group. 
\begin{enumerate}
\item Consider a link state $p\in GP_i$. Then $GJ_p$ is a left ideal of $G\mathcal{RB}_n\left(\delta , \varepsilon\right)$.
\item Consider a link state $p\in BP_i$. Then $BJ_p$ is a left ideal of $B\mathcal{RB}_n\left(\delta , \varepsilon\right)$.
\item Consider a link state $p\in GBP_i$. Then $GBJ_p$ is a left ideal of $GB\mathcal{RB}_n\left(\delta , \varepsilon\right)$.
\end{enumerate}
\end{lem}
\begin{proof}
Given two $G$-rook-Brauer diagrams, $x$ and $y$, the right link state of $x\cdot y$ is obtained from the right link state of $y$ by a sequence of $G$-splices, deletions and label replacements. The remaining two statements are proved similarly.
\end{proof}

\section{Technical results}
\label{technical-results-sec}
In this section we prove the necessary generalizations of \cite[Theorem 1.7]{Boyde} and \cite[Lemma 6.1]{Boyde}. These will allow us to prove results about generalized rook-Brauer algebras, Brauer algebras and Temperley-Lieb diagrams.

\subsection{Equivariant case}

\begin{prop}
\label{equiv-tech-prop}
Let $G$ be an abelian group. Let $0\leqslant l \leqslant m \leqslant n-1$. Let $A$ be a subalgebra of $G\mathcal{RB}(\delta , \varepsilon)$, such that 
\begin{itemize}
\item as a $k$-module, $A$ is free on a subset of the $G$-rook-Brauer $n$-diagrams, and
\item for $i$ in the range $l\leqslant i \leqslant m$, for each link state $p\in GP_i$, if $A$ contains at least one diagram with right link state $p$, then $A$ contains an idempotent $e_p$ such that in $A$ we have an equality of left ideals $A\cdot e_p = A\cap GJ_p$.
\end{itemize}
Then we have a chain of isomorphisms 
\[\mathrm{Tor}_{\star}^{A/\left(A\cap I_{l-1}\right)}\left(\mathbbm{1} , \mathbbm{1}\right) \cong \mathrm{Tor}_{\star}^{A/\left(A\cap I_{l}\right)}\left(\mathbbm{1} , \mathbbm{1}\right)\cong \cdots \cong \mathrm{Tor}_{\star}^{A/\left(A\cap I_{m}\right)}\left(\mathbbm{1} , \mathbbm{1}\right) .\]
\end{prop}
\begin{proof}
The proof follows the same method as the proof of \cite[Theorem 1.7]{Boyde} given in Section 5 of that paper. The only change required is that we use $G$-splices instead of splices.
\end{proof}

\begin{defn}
\label{fp-diag-defn}
Let $p\in GP_i$ be a link state with no missing edges. Let $f_p$ be the diagram obtained as follows:
\begin{enumerate}
\item replace the labels on all defects by the identity element in $G$;
\item complete to the diagram such that the underlying, unlabelled diagram of both the left and right link states is the underlying, unlabelled diagram of $p$;
\item for each right-to-right connection, if the label on the edge is $g \in G$, we set the label on the corresponding left-to-left connection by $g^{-1}$.
\end{enumerate} 
\end{defn}

\begin{lem}
\label{equiv-tech-lem}
Let $y$ be a diagram in $G\mathcal{B}_n(\delta) \cap GJ_p$. Then $y\cdot f_p = \delta^{\frac{1}{2}(n-i)}y$.
\end{lem}
\begin{proof}
The proof follows the same method as \cite[Lemma 6.1]{Boyde}, with $f_p$ playing the role of $d_p$. We note that our labelling convention in forming the diagram $f_p$ ensures that the label on any loops formed in the composite $y\cdot f_p$ is the identity element in $G$.
\end{proof}

\subsection{Braided case}

\begin{prop}
\label{braided-tech-prop}
Let $0\leqslant l \leqslant m \leqslant n-1$. Let $A$ be a subalgebra of $B\mathcal{RB}(\delta , \varepsilon)$, such that 
\begin{itemize}
\item as a $k$-module, $A$ is free on a subset of the braided rook-Brauer $n$-diagrams, and
\item for $i$ in the range $l\leqslant i \leqslant m$, for each link state $p\in BP_i$, if $A$ contains at least one diagram with right link state $p$, then $A$ contains an idempotent $e_p$ such that in $A$ we have an equality of left ideals $A\cdot e_p = A\cap BJ_p$.
\end{itemize}
Then we have a chain of isomorphisms 
\[\mathrm{Tor}_{\star}^{A/\left(A\cap I_{l-1}\right)}\left(\mathbbm{1} , \mathbbm{1}\right) \cong \mathrm{Tor}_{\star}^{A/\left(A\cap I_{l}\right)}\left(\mathbbm{1} , \mathbbm{1}\right)\cong \cdots \cong \mathrm{Tor}_{\star}^{A/\left(A\cap I_{m}\right)}\left(\mathbbm{1} , \mathbbm{1}\right) .\]
\end{prop}
\begin{proof}
The proof follows the same method as the proof of \cite[Theorem 1.7]{Boyde} given in Section 5 of that paper. The only change required is that we use braided splices instead of splices.
\end{proof}

\begin{defn}
\label{dp-diag-defn}
Let $p\in BP_i$ be a link state with no missing edges. Let $d_p$ be the diagram such that the left and right link states are both $p$ and all of whose left-to-right connections are horizontal edges. 
\end{defn}

\begin{lem}
\label{braided-tech-lem}
Let $y$ be a diagram in $B\mathcal{B}_n(\delta) \cap BJ_p$. Then $y\cdot d_p = \delta^{\frac{1}{2}(n-i)}y$.
\end{lem}
\begin{proof}
The proof follows the same method as \cite[Lemma 6.1]{Boyde}.
\end{proof}

\subsection{$G$-braided case}

\begin{prop}
\label{G-braided-tech-prop}
Let $G$ be an abelian group. Let $0\leqslant l \leqslant m \leqslant n-1$. Let $A$ be a subalgebra of $GB\mathcal{RB}(\delta , \varepsilon)$, such that 
\begin{itemize}
\item as a $k$-module, $A$ is free on a subset of the braided rook-Brauer $n$-diagrams, and
\item for $i$ in the range $l\leqslant i \leqslant m$, for each link state $p\in GBP_i$, if $A$ contains at least one diagram with right link state $p$, then $A$ contains an idempotent $e_p$ such that in $A$ we have an equality of left ideals $A\cdot e_p = A\cap GBJ_p$.
\end{itemize}
Then we have a chain of isomorphisms 
\[\mathrm{Tor}_{\star}^{A/\left(A\cap I_{l-1}\right)}\left(\mathbbm{1} , \mathbbm{1}\right) \cong \mathrm{Tor}_{\star}^{A/\left(A\cap I_{l}\right)}\left(\mathbbm{1} , \mathbbm{1}\right)\cong \cdots \cong \mathrm{Tor}_{\star}^{A/\left(A\cap I_{m}\right)}\left(\mathbbm{1} , \mathbbm{1}\right) .\]
\end{prop}
\begin{proof}
The proof follows the same method as the proof of \cite[Theorem 1.7]{Boyde} given in Section 5 of that paper. The only change required is that we use braided $G$-splices instead of splices.
\end{proof}

\begin{defn}
\label{hp-diag-defn}
Let $p\in GBP_i$ be a link state with no missing edges. Let $h_p$ be the diagram obtained as follows:
\begin{enumerate}
\item replace the labels on all defects by the identity element in $G$;
\item complete to the diagram such that the underlying, unlabelled diagram of both the left and right link states is the underlying, unlabelled diagram of $p$;
\item for each right-to-right connection, if the label on the edge is $g \in G$, we set the label on the corresponding left-to-left connection by $g^{-1}$.
\end{enumerate} 
\end{defn}

\begin{lem}
\label{G-braided-tech-lem}
Let $y$ be a diagram in $GB\mathcal{B}_n(\delta) \cap BGJ_p$. Then $y\cdot h_p = \delta^{\frac{1}{2}(n-i)}y$.
\end{lem}
\begin{proof}
The proof follows the same method as \cite[Lemma 6.1]{Boyde}, with $h_p$ playing the role of $d_p$. We note that our labelling convention in forming the diagram $h_p$ ensures that the label on any loops formed in the composite $y\cdot h_p$ is the identity element in $G$.
\end{proof}

\section{Group algebra retracts}
\label{group-alg-sec}
In this section we prove the analogues of Boyde's results on group algebra retractions \cite[Section 2]{Boyde}. We treat the equivariant rook algebras and the equivariant rook-Brauer algebras separately as the former hold for any group whereas the latter require a restriction to abelian groups.

\begin{defn}
Fix an natural number $n$.
\begin{enumerate}
\item Let $G$ be an abelian group. Let $BS_{max}$ (resp. $GS_{max}$ and $GBS_{max}$) denote the set of braided rook-Brauer $n$-diagrams (resp. $G$-rook-Brauer $n$-diagrams and $G$-braided rook-Brauer $n$-diagrams) with the maximum number of left-to-right connections.
\item Let $G$ be a group. Let $GR_{max}$ (resp. $GBR_{max}$) denote the set of  $G$-rook $n$-diagrams (resp. $G$-braided rook $n$-diagrams) with the maximum number of left-to-right connections.
\end{enumerate}
\end{defn}

\begin{prop}
\label{group-alg-retract-prop1}
Let $\delta,\, \varepsilon \in k$. Let $n\in \mathbb{N}$.
\begin{enumerate}
\item Let $G$ be an abelian group.
\begin{enumerate}
\item the sets $BS_{max}$, $GS_{max}$ and $GBS_{max}$ form multiplicatively closed subsets of the algebras $B\mathcal{RB}_n(\delta, \varepsilon)$, $G\mathcal{RB}_n(\delta, \varepsilon)$ and $GB\mathcal{RB}_n(\delta, \varepsilon)$ respectively;
\item these subsets are canonically isomorphic to the braid group $B_n$, the group $G^n\rtimes \Sigma_n$ and the group $G^n \rtimes B_n$ respectively.
\end{enumerate}
\item Let $G$ be a group. 
\begin{enumerate}
\item the sets $GR_{max}$ and $GBR_{max}$ form multiplicatively closed subsets of the algebras $G\mathcal{R}_n(\varepsilon)$ and $GB\mathcal{R}_n(\varepsilon)$ respectively;
\item these subsets are canonically isomorphic to the group $G^n\rtimes \Sigma_n$ and the group $G^n \rtimes B_n$ respectively.
\end{enumerate}
\end{enumerate}
\end{prop}
\begin{proof}
Each statement follows by sending a group element to its pictorial representation, similarly to \cite[Proposition 2.1]{Boyde}. Composition of group elements then corresponds to composition of $n$-diagrams.
\end{proof}

\begin{cor}
We have the following corollaries.
\label{subgroup-cor}
\begin{enumerate}
\item Let $G$ be an abelian group. Let $A$ be a subalgebra of $B\mathcal{RB}_n(\delta, \varepsilon)$ (resp. $G\mathcal{RB}_n(\delta, \varepsilon)$ and $GB\mathcal{RB}_n(\delta, \varepsilon)$), which is free on a subset of $n$-diagrams and contains at least one diagram having $n$ left-to-right connections. The set of diagrams in $A$ having $n$ left-to-right connections is multiplicatively closed and is canonically isomorphic to some subgroup $H\leqslant B_n$ (resp. $G^n\rtimes \Sigma_n$ and $G^n\rtimes B_n$).
\item Let $G$ be a group. Let $A$ be a subalgebra of $G\mathcal{R}_n(\varepsilon)$ (resp. $GB\mathcal{R}_n(\varepsilon)$), which is free on a subset of $n$-diagrams and contains at least one diagram having $n$ left-to-right connections. The set of diagrams in $A$ having $n$ left-to-right connections is multiplicatively closed and is canonically isomorphic to some subgroup $H\leqslant G^n\rtimes \Sigma_n$ (resp. $G^n\rtimes B_n$).
\end{enumerate}
\end{cor}
\begin{proof}
As for \cite[Corollary 2.2]{Boyde}, we note that in each case the sets $BS_{max}$, $GS_{max}$, $GBS_{max}$, $GR_{max}$ and $GBR_{max}$ have non-trivial intersection with $A$. Therefore, under the isomorphisms of Proposition \ref{group-alg-retract-prop1}, the intersections $A\cap BS_{max}$, $A\cap GS_{max}$, $A\cap GBS_{max}$, $A\cap GR_{max}$ and $A\cap GBR_{max}$  are identified as subgroups of $B_n$, $G^n\rtimes \Sigma_n$ and $G^n\rtimes B_n$ as appropriate.
\end{proof}

\begin{prop}
We have the following propositions.
\label{group-alg-retract-prop2}
\begin{enumerate}
\item Let $G$ be an abelian group. Let $A$ be a subalgebra of $B\mathcal{RB}_n(\delta, \varepsilon)$ (resp. $G\mathcal{RB}_n(\delta, \varepsilon)$ and $GB\mathcal{RB}_n(\delta, \varepsilon)$), such that
\begin{itemize}
\item $A$ is free on a subset of the $n$-diagrams and
\item $A$ contains at least one diagram with $n$ left-to-right connections.
\end{itemize} 
Let $A_{max}$ denote the $k$-span of the diagrams in $A$ with $n$ left-to-right connections. The following properties are satisfied:
\begin{enumerate}
\item $A_{max}$ is a subalgebra of $A$ and it is canonically isomorphic to the group ring $k[H]$ for some $H\leqslant B_n$ (resp. $G^n\rtimes \Sigma_n$ and $G^n\rtimes B_n$),
\item $A_{max}$ is isomorphic to the quotient $A/(A \cap I_{n-1})$. Hence $A_{max}$ is a retract of $A$.
\end{enumerate}
\item Let $G$ be a group. Let $A$ be a subalgebra of $G\mathcal{R}_n(\varepsilon)$ (resp. $GB\mathcal{R}_n(\varepsilon)$), such that
\begin{itemize}
\item $A$ is free on a subset of the $n$-diagrams and
\item $A$ contains at least one diagram with $n$ left-to-right connections.
\end{itemize} 
Let $A_{max}$ denote the $k$-span of the diagrams in $A$ with $n$ left-to-right connections. The following properties are satisfied:
\begin{enumerate}
\item $A_{max}$ is a subalgebra of $A$ and it is canonically isomorphic to the group ring $k[H]$ for some $H\leqslant G^n\rtimes \Sigma_n$ (resp. $G^n\rtimes B_n$),
\item $A_{max}$ is isomorphic to the quotient $A/(A \cap I_{n-1})$. Hence $A_{max}$ is a retract of $A$.
\end{enumerate}
\end{enumerate}
\end{prop}
\begin{proof}
In each case, the first part follows directly from Corollary \ref{subgroup-cor}. The second part follows the method of \cite[Proposition 2.3]{Boyde}.
\end{proof}

\section{Homology of generalized rook algebras}
\label{gen-rook-hom-sec}
In this section we study the homology of generalized rook algebras when the parameter $\varepsilon$ is invertible. We show that the homology of a braided rook algebra coincides with the group homology of the braid groups. When $G$ is a group, we show that the homology of the $G$-rook algebras and the $G$-braided rook algebras coincide with the group homology of the groups $G^n\rtimes \Sigma_n$ and $G^n\rtimes B_n$ respectively. We deduce homological stability for these generalized rook algebras.

\subsection{Idempotents}

\begin{defn}
For $0\leqslant i \leqslant n$, let $\rho_i$ denote the element in a generalized rook algebra obtained from the identity diagram by deleting the $i^{th}$ edge.
\end{defn}

\begin{lem}
\label{idempotent-lemma}
If $\varepsilon \in k$ is invertible, then $\varepsilon^{-1}\rho_i$ is idempotent.
\end{lem}
\begin{proof}
We observe that 
\[\left(\varepsilon^{-1}\rho_i\right)^2= \varepsilon^{-2} \rho_i^2= \varepsilon^{-2}\varepsilon \rho_i=\varepsilon^{-1}\rho_i\]
as required.
\end{proof}

\begin{lem}
\label{ideal-basis-lem}
Let $X$ be a generalized rook algebra. For any $\varepsilon \in k$, the left ideal $X\cdot \rho_i$ of $X$ has $k$-basis consisting of those diagrams where the $i^{th}$ node on the right is not connected to anything on the left.
\end{lem}
\begin{proof}
The proof follows the method of \cite[Lemma 4.3]{Boyde}.
\end{proof}

\subsection{Homology of generalized rook algebras}

We are now in the position to prove our results about the homology of generalized rook algebras. The proofs for the three statements are similar. We will prove the result for braided rook algebras and simply state the equivariant cases.

\begin{thm}
\label{braid-rook-thm}
Let $\varepsilon \in k$ be invertible. Let $B\mathcal{R}_n(\varepsilon)$ be the braided rook algebra. Let $B_n$ denote the braid group on $n$ strands and let $k[B_n]$ denote the group algebra. There exists an isomorphism of graded $k$-modules
\[\mathrm{Tor}_{\star}^{B\mathcal{R}_n(\varepsilon)}\left(\mathbbm{1},\mathbbm{1}\right) \cong \mathrm{Tor}_{\star}^{k[B_n]}\left(\mathbbm{1},\mathbbm{1}\right)=H_{\star}\left(B_n , k\right).\]
\end{thm}
\begin{proof}
To simplify notation, we will write $X=B\mathcal{R}_n(\varepsilon)$.

Let $I=X(\varepsilon^{-1}\rho_1) + \cdots + X(\varepsilon^{-1}\rho_n)$ be the left ideal generated by the idempotents of Lemma \ref{idempotent-lemma}. The elements $\varepsilon^{-1}\rho_i$ and $\varepsilon^{-1}\rho_j$ commute for all $i$ and $j$. Furthermore, we note that $I$ acts trivially on the trivial module $\mathbbm{1}$. We can therefore apply \cite[Theorem 4.1]{Boyde} to obtain an isomorphism
\[\mathrm{Tor}_{\star}^{B\mathcal{R}_n(\varepsilon)}\left(\mathbbm{1},\mathbbm{1}\right) \cong  \mathrm{Tor}_{\star}^{B\mathcal{R}_n(\varepsilon)/I}\left(\mathbbm{1},\mathbbm{1}\right).\]

By definition, $I$ has a $k$-basis consisting of the diagrams with at least one missing left-to-right connection. Therefore the quotient $B\mathcal{R}_n(\varepsilon)/I$ has a $k$-basis consisting of the diagrams with no missing left-to-right connections.

By Proposition \ref{group-alg-retract-prop2}, we have an isomorphism
\[\frac{B\mathcal{R}_n(\varepsilon)}{I} \cong k[B_n]\]
and we see that the trivial module for $B\mathcal{R}_n(\varepsilon)$ restricts to the trivial module for $k[B_n]$. Therefore we have an isomorphism 
\[\mathrm{Tor}_{\star}^{B\mathcal{R}_n(\varepsilon)}\left(\mathbbm{1},\mathbbm{1}\right) \cong \mathrm{Tor}_{\star}^{k[B_n]}\left(\mathbbm{1},\mathbbm{1}\right)\]
as required.
\end{proof}

\begin{thm}
\label{G-symm-rook-thm}
Let $G$ be a group. Let $k$ be a commutative $G$-ring. Let $\varepsilon \in k$ be invertible. Let $G\mathcal{R}_n(\varepsilon)$ be the $G$-rook algebra. Let $G^n\rtimes \Sigma_n$ denote the semi-direct product, where $\Sigma_n$ acts on $G^n$ by permuting the factors. Let $k[G^n\rtimes \Sigma_n]$ denote the group algebra. There exists an isomorphism of graded $k$-modules
\[\mathrm{Tor}_{\star}^{G\mathcal{R}_n(\varepsilon)}\left(\mathbbm{1},\mathbbm{1}\right) \cong \mathrm{Tor}_{\star}^{k[G^n\rtimes \Sigma_n]}\left(\mathbbm{1},\mathbbm{1}\right)=H_{\star}\left(G^n\rtimes \Sigma_n , k\right).\]
\end{thm}
\begin{proof}
The proof of this statements follows the method Theorem \ref{braid-rook-thm} by replacing $B_n$ with $G^n\rtimes \Sigma_n$.
\end{proof}

\begin{rem}
In particular, take $G=C_2$, the cyclic group of order two. In this case, we have $C_2^n\rtimes \Sigma_n = H_n$, the $n^{th}$ hyperoctahedral group, and Theorem \ref{G-symm-rook-thm} gives us an isomorphism 
\[\mathrm{Tor}_{\star}^{C_2\mathcal{R}_n(\varepsilon)}\left(\mathbbm{1},\mathbbm{1}\right) \cong \mathrm{Tor}_{\star}^{k[H_n]}\left(\mathbbm{1},\mathbbm{1}\right)=H_{\star}\left(H_n, k\right).\]
\end{rem}

\begin{rem}
We observe that if we take $G$ to be the trivial group in Theorem \ref{G-symm-rook-thm} then we recover \cite[Theorem 4.4]{Boyde}.
\end{rem}

\begin{thm}
\label{G-braid-rook-thm}
Let $G$ be a group. Let $k$ be a commutative $G$-ring. Let $\varepsilon \in k$ be invertible. Let $GB\mathcal{R}_n(\varepsilon)$ be the $G$-braided rook algebra. Let $G^n\rtimes B_n$ denote the semi-direct product, where $B_n$ acts on $G^n$ by permuting the factors. Let $k[G^n\rtimes B_n]$ denote the group algebra. There exists an isomorphism of graded $k$-modules
\[\mathrm{Tor}_{\star}^{GB\mathcal{R}_n(\varepsilon)}\left(\mathbbm{1},\mathbbm{1}\right) \cong \mathrm{Tor}_{\star}^{k[G^n\rtimes B_n]}\left(\mathbbm{1},\mathbbm{1}\right)=H_{\star}\left(G^n\rtimes B_n , k\right).\]
\end{thm}
\begin{proof}
The proof of this statements follows the method Theorem \ref{braid-rook-thm} by replacing $B_n$ with $G^n\rtimes B_n$.
\end{proof}

\subsection{Homological stability}
We can use the theorems from the previous subsection to deduce homological stability for generalized rook algebras.

\begin{cor}
Let $\varepsilon \in k$ be invertible. The morphism
\[\mathrm{Tor}_{q}^{B\mathcal{R}_{n-1}(\varepsilon)}\left(\mathbbm{1},\mathbbm{1}\right) \rightarrow \mathrm{Tor}_{q}^{B\mathcal{R}_{n}(\varepsilon)}\left(\mathbbm{1},\mathbbm{1}\right)\]
induced by the inclusion map $B\mathcal{R}_{n-1}(\varepsilon) \rightarrow B\mathcal{R}_{n}(\varepsilon)$ is an isomorphism for $n\geqslant 2q+1$.
\end{cor}
\begin{proof}
This follows from Theorem \ref{braid-rook-thm} and Arnold's result for the homological stability of the braid groups (originally published as \cite{Arnold1} and available in \cite{Arnold2}).
\end{proof}

\begin{cor}
Let $G$ be a group. Let $k$ be a commutative $G$-ring. Let $\varepsilon \in k$ be invertible. The morphism
\[\mathrm{Tor}_{q}^{G\mathcal{R}_{n-1}(\varepsilon)}\left(\mathbbm{1},\mathbbm{1}\right) \rightarrow \mathrm{Tor}_{q}^{G\mathcal{R}_{n}(\varepsilon)}\left(\mathbbm{1},\mathbbm{1}\right)\]
induced by the inclusion map $G\mathcal{R}_{n-1}(\varepsilon) \rightarrow G\mathcal{R}_{n}(\varepsilon)$ is an isomorphism for $n\geqslant 2q+2$.
\end{cor}
\begin{proof}
This follows from Theorem \ref{G-symm-rook-thm} together with Hatcher and Wahl's result for the homological stability of the groups $G^n\rtimes \Sigma_n$ \cite[Proposition 1.6]{HW}.
\end{proof}

\begin{cor}
Let $G$ be a group. Let $k$ be a commutative $G$-ring. Let $\varepsilon \in k$ be invertible. The morphism
\[\mathrm{Tor}_{q}^{GB\mathcal{R}_{n-1}(\varepsilon)}\left(\mathbbm{1},\mathbbm{1}\right) \rightarrow \mathrm{Tor}_{q}^{GB\mathcal{R}_{n}(\varepsilon)}\left(\mathbbm{1},\mathbbm{1}\right)\]
induced by the inclusion map $GB\mathcal{R}_{n-1}(\varepsilon) \rightarrow GB\mathcal{R}_{n}(\varepsilon)$ is an isomorphism for $n\geqslant 2q+2$.
\end{cor}
\begin{proof}
This follows from Theorem \ref{G-braid-rook-thm} together with Hatcher and Wahl's result for the homological stability of the groups $G^n\rtimes B_n$ \cite[Proposition 1.7]{HW}.
\end{proof}

\section{Homology of generalized Brauer algebras}
\label{gen-brauer-sec}

In this section we will study the homology of generalized Brauer algebras and generalized Temperley-Lieb algebras when the parameter $\delta$ is invertible. We will show that the homology of a braided Brauer algebra coincides with the homology of the braid groups. When $G$ is an abelian group, we will show that the homology of the $G$-Brauer algebras and $G$-braided Brauer algebras coincide with the group homology of the groups $G^n\rtimes \Sigma_n$ and $G^n\rtimes B_n$ respectively. We deduce homological stability results for these generalized Brauer algebras. We show that the homology of the $G$-Temperley-Lieb algebras coincides with the group homology of $G^n$. We can see Theorems \ref{G-brauer-thm}, \ref{braided-brauer-thm} and \ref{G-braided-brauer-thm} as equivariant and braided versions of \cite[Theorem A]{BHP}. In particular, if we take $G$ to be the trivial group in Theorem \ref{G-brauer-thm}, we recover \cite[Theorem A]{BHP} and \cite[Theorem 6.5]{Boyde}.

\subsection{Homology of generalized Brauer algebras}

\begin{thm}
\label{G-brauer-thm}
Let $G$ be an abelian group. If $\delta \in k$ is invertible, there is an isomorphism of graded $k$-modules
\[\mathrm{Tor}_{\star}^{G\mathcal{B}_n(\delta)}\left(\mathbbm{1} , \mathbbm{1}\right) \cong \mathrm{Tor}_{\star}^{k[G^n\rtimes \Sigma_n]}\left(\mathbbm{1},\mathbbm{1}\right) = H_{\star}\left(G^n\rtimes \Sigma_n , k\right).\]
\end{thm}
\begin{proof}
This follows the same method as \cite[Theorem 6.5]{Boyde}, with Proposition \ref{equiv-tech-prop} playing the role of \cite[Theorem 1.7]{Boyde}, the diagram $f_p$ from Definition \ref{fp-diag-defn} playing the role of $d_p$ and Lemma \ref{equiv-tech-lem} playing the role of \cite[Lemma 6.1]{Boyde}.
\end{proof}

\begin{thm}
\label{braided-brauer-thm}
If $\delta \in k$ is invertible, there is an isomorphism of graded $k$-modules
\[\mathrm{Tor}_{\star}^{B\mathcal{B}_n(\delta)}\left(\mathbbm{1} , \mathbbm{1}\right) \cong \mathrm{Tor}_{\star}^{k[B_n]}\left(\mathbbm{1},\mathbbm{1}\right) = H_{\star}\left(B_n , k\right).\]
\end{thm}
\begin{proof}
This follows the same method as \cite[Theorem 6.5]{Boyde}, with Proposition \ref{braided-tech-prop} playing the role of \cite[Theorem 1.7]{Boyde}, the diagram $d_p$ from Definition \ref{dp-diag-defn} playing the role of Boyde's $d_p$, and Lemma \ref{braided-tech-lem} playing the role of \cite[Lemma 6.1]{Boyde}.
\end{proof}

\begin{thm}
\label{G-braided-brauer-thm}
Let $G$ be an abelian group. If $\delta \in k$ is invertible, there is an isomorphism of graded $k$-modules
\[\mathrm{Tor}_{\star}^{GB\mathcal{B}_n(\delta)}\left(\mathbbm{1} , \mathbbm{1}\right) \cong \mathrm{Tor}_{\star}^{k[G^n\rtimes B_n]}\left(\mathbbm{1},\mathbbm{1}\right) = H_{\star}\left(G^n\rtimes B_n , k\right).\]
\end{thm}
\begin{proof}
This follows the same method as \cite[Theorem 6.5]{Boyde}, with Proposition \ref{G-braided-tech-prop} playing the role of \cite[Theorem 1.7]{Boyde}, the diagram $h_p$ from Definition \ref{hp-diag-defn} playing the role of $d_p$ and Lemma \ref{G-braided-tech-lem} playing the role of \cite[Lemma 6.1]{Boyde}.
\end{proof}

\subsection{Homological stability}
We can use the theorems from the previous subsection to deduce homological stability for generalized Brauer algebras.

\begin{cor}
\label{G-brauer-cor}
Let $G$ be an abelian group. Let $\delta \in k$ be invertible. The map
\[\mathrm{Tor}_{q}^{G\mathcal{B}_{n-1}(\delta)}\left(\mathbbm{1} , \mathbbm{1}\right)\rightarrow \mathrm{Tor}_{q}^{G\mathcal{B}_n(\delta)}\left(\mathbbm{1} , \mathbbm{1}\right)\]
induced by the inclusion $G\mathcal{B}_{n-1}(\delta) \rightarrow G\mathcal{B}_{n}(\delta)$ is an isomorphism for $n\geqslant 2q+2$.
\end{cor}
\begin{proof}
This follows from Theorem \ref{G-brauer-thm} together with Hatcher and Wahl's result for the homological stability of the groups $G^n\rtimes \Sigma_n$ \cite[Proposition 1.6]{HW}.
\end{proof}

\begin{cor}
\label{braided-brauer-cor}
Let $\delta \in k$ be invertible. The map
\[\mathrm{Tor}_{q}^{B\mathcal{B}_{n-1}(\delta)}\left(\mathbbm{1} , \mathbbm{1}\right)\rightarrow \mathrm{Tor}_{q}^{B\mathcal{B}_n(\delta)}\left(\mathbbm{1} , \mathbbm{1}\right)\]
induced by the inclusion $B\mathcal{B}_{n-1}(\delta) \rightarrow B\mathcal{B}_{n}(\delta)$ is an isomorphism for $n\geqslant 2q+1$.
\end{cor}
\begin{proof}
This follows from Theorem \ref{braided-brauer-thm} together with Arnold's result for the homological stability of the braid groups (\cite{Arnold1} and \cite{Arnold2}).
\end{proof}

\begin{cor}
\label{G-braided-brauer-cor}
Let $G$ be an abelian group. Let $\delta \in k$ be invertible. The map
\[\mathrm{Tor}_{q}^{GB\mathcal{B}_{n-1}(\delta)}\left(\mathbbm{1} , \mathbbm{1}\right)\rightarrow \mathrm{Tor}_{q}^{GB\mathcal{B}_n(\delta)}\left(\mathbbm{1} , \mathbbm{1}\right)\]
induced by the inclusion $GB\mathcal{B}_{n-1}(\delta) \rightarrow GB\mathcal{B}_{n}(\delta)$ is an isomorphism for $n\geqslant 2q+2$.
\end{cor}
\begin{proof}
This follows from Theorem \ref{G-braided-brauer-thm} together with Hatcher and Wahl's result for the homological stability of the groups $G^n\rtimes B_n$ \cite[Proposition 1.7]{HW}. 
\end{proof}

\subsection{Homology of generalized Temperley-Lieb algebras}

\begin{thm}
\label{G-temp-lieb-thm}
Let $G$ be an abelian group. If $\delta \in k$ is invertible, then there is an isomorphism of graded $k$-modules
\[\mathrm{Tor}_{\star}^{G\mathcal{TL}_n(\delta)}\left(\mathbbm{1} , \mathbbm{1}\right) \cong \mathrm{Tor}_{\star}^{k[G^n]}\left(\mathbbm{1},\mathbbm{1}\right).\]
\end{thm}
\begin{proof}
This follows the same method of proof as \cite[Theorem 6.6]{Boyde}, using Proposition \ref{equiv-tech-prop} in place of \cite[Theorem 1.7]{Boyde}, $f_p$ in place of $d_p$ and noting that we have an isomorphism
\[\frac{G\mathcal{TL}_n}{G\mathcal{TL}_n \cap I_{n-1}}\cong k[G^n].\]
\end{proof}

\begin{rem}
If we take $G$ to be the trivial group then we recover \cite[Theorem A]{BH1} and \cite[Theorem 6.6]{Boyde}.
\end{rem}

\begin{cor}
Taking $n=1$ we obtain an isomorphism of graded $k$-modules
\[\mathrm{Tor}_{\star}^{G\mathcal{TL}_1(\delta)}\left(\mathbbm{1} , \mathbbm{1}\right) \cong H_{\star}(G,k). \]
\end{cor}
\begin{proof}
This follows directly from Theorem \ref{G-temp-lieb-thm} by identifying the right hand side with the group homology of $G$.
\end{proof}

\begin{rem}
One might wonder whether the invertibility of $\delta$ is important in this corollary. After all, when $n=1$ we cannot form any loops so the parameter $\delta$ goes unused. We will show in Corollary \ref{non-invertible-temp-lieb-cor} that we may drop the condition on the invertibility of $\delta$ for this corollary.  
\end{rem}

\section{Homology of generalized rook-Brauer algebras}
\label{gen-RB-sec}

In this section we will show that the homology of a generalized rook-Brauer algebra coincides with the homology of the generalized Brauer subalgebra when the parameter $\varepsilon$ is invertible. As a consequence we deduce some homological stability results for these generalized rook-Brauer algebras. 

\begin{thm}
\label{gen-rook-brauer-thm}
Let $\varepsilon\in k$ be invertible. Then for any $\delta \in k$ there is an isomorphism of graded $k$-modules
\[\mathrm{Tor}_{\star}^{B\mathcal{RB}_n(\delta, \varepsilon)}\left(\mathbbm{1},\mathbbm{1}\right) \cong \mathrm{Tor}_{\star}^{B\mathcal{B}_n(\delta)}\left(\mathbbm{1},\mathbbm{1}\right).\] 
\end{thm}
\begin{proof}
Since $\varepsilon$ is invertible, we have the idempotents of Lemma \ref{idempotent-lemma}.
Write $X=B\mathcal{RB}_n(\delta, \varepsilon)$. Let $I=X(\varepsilon^{-1}\rho_1) + \cdots + X(\varepsilon^{-1}\rho_n)$. Applying \cite[Theorem 4.1]{Boyde} gives the result.
\end{proof}

\begin{cor}
If $\delta \in k$ is also invertible, then there is an isomorphism of graded $k$-modules
\[\mathrm{Tor}_{\star}^{B\mathcal{RB}_n(\delta, \varepsilon)}\left(\mathbbm{1},\mathbbm{1}\right) \cong H_{\star}\left(B_n , k\right)\]
and the map 
\[\mathrm{Tor}_{q}^{B\mathcal{RB}_{n-1}(\delta, \varepsilon)}\left(\mathbbm{1},\mathbbm{1}\right) \rightarrow \mathrm{Tor}_{q}^{B\mathcal{RB}_n(\delta, \varepsilon)}\left(\mathbbm{1},\mathbbm{1}\right)\]
is an isomorphism for $n\geqslant 2q+1$.
\end{cor}
\begin{proof}
This follows from Theorem \ref{braided-brauer-thm} and Corollary \ref{braided-brauer-cor}.
\end{proof}

\begin{thm}
Let $G$ be an abelian group. Let $\varepsilon\in k$ be invertible. Then for any $\delta \in k$ there is an isomorphism of graded $k$-modules
\[\mathrm{Tor}_{\star}^{G\mathcal{RB}_n(\delta, \varepsilon)}\left(\mathbbm{1},\mathbbm{1}\right) \cong \mathrm{Tor}_{\star}^{G\mathcal{B}_n(\delta)}\left(\mathbbm{1},\mathbbm{1}\right).\] 
\end{thm}
\begin{proof}
The proof is similar to Theorem \ref{gen-rook-brauer-thm}.
\end{proof}

\begin{cor}
If $\delta \in k$ is also invertible, then there is an isomorphism of graded $k$-modules
\[\mathrm{Tor}_{\star}^{G\mathcal{RB}_n(\delta, \varepsilon)}\left(\mathbbm{1},\mathbbm{1}\right) \cong H_{\star}\left(G^n\rtimes \Sigma_n , k\right)\]
and the map 
\[\mathrm{Tor}_{q}^{G\mathcal{RB}_{n-1}(\delta, \varepsilon)}\left(\mathbbm{1},\mathbbm{1}\right) \rightarrow \mathrm{Tor}_{q}^{G\mathcal{RB}_n(\delta, \varepsilon)}\left(\mathbbm{1},\mathbbm{1}\right)\]
is an isomorphism for $n\geqslant 2q+2$.
\end{cor}
\begin{proof}
This follows from Theorem \ref{G-brauer-thm} and Corollary \ref{G-brauer-cor}.
\end{proof}

\begin{thm}
Let $G$ be an abelian group. Let $\varepsilon\in k$ be invertible. Then for any $\delta \in k$ there is an isomorphism of graded $k$-modules
\[\mathrm{Tor}_{\star}^{GB\mathcal{RB}_n(\delta, \varepsilon)}\left(\mathbbm{1},\mathbbm{1}\right) \cong \mathrm{Tor}_{\star}^{GB\mathcal{B}_n(\delta)}\left(\mathbbm{1},\mathbbm{1}\right).\] 
\end{thm}
\begin{proof}
The proof is similar to Theorem \ref{gen-rook-brauer-thm}.
\end{proof}

\begin{cor}
If $\delta \in k$ is also invertible, then there is an isomorphism of graded $k$-modules
\[\mathrm{Tor}_{\star}^{GB\mathcal{RB}_n(\delta, \varepsilon)}\left(\mathbbm{1},\mathbbm{1}\right) \cong H_{\star}\left(G^n\rtimes B_n , k\right)\]
and the map 
\[\mathrm{Tor}_{q}^{GB\mathcal{RB}_{n-1}(\delta, \varepsilon)}\left(\mathbbm{1},\mathbbm{1}\right) \rightarrow \mathrm{Tor}_{q}^{GB\mathcal{RB}_n(\delta, \varepsilon)}\left(\mathbbm{1},\mathbbm{1}\right)\]
is an isomorphism for $n\geqslant 2q+2$.
\end{cor}
\begin{proof}
This follows from Theorem \ref{G-braided-brauer-thm} and Corollary \ref{G-braided-brauer-cor}.
\end{proof}

\section{Homology of Motzkin algebras}
\label{motzkin-sec}
In this section we investigate the homology of Motzkin algebras in both the non-equivariant and equivariant cases. To the best of the author's knowledge this has not been studied before. In the non-equivariant case we show that when the parameter $\varepsilon$ is invertible, the homology of the Motzkin algebras coincides with the homology of the Temperley-Lieb algebras. As a consequence, we deduce vanishing results and homological stability results by combining our theorem with the results of \cite{BH1}. In the equivariant case, for an abelian group $G$, we show that the homology of the $G$-Motzkin algebras coincides with the homology of the $G$-Temperley-Lieb algebras.

\subsection{The non-equivariant case}

\begin{thm}
\label{motzkin-thm1}
Let $\varepsilon \in k$ be invertible. Then for any $\delta \in k$ there is an isomorphism of graded $k$-modules
\[\mathrm{Tor}_{\star}^{\mathcal{M}_{n}(\delta, \varepsilon)}\left(\mathbbm{1},\mathbbm{1}\right) \cong\mathrm{Tor}_{\star}^{\mathcal{TL}_n(\delta)}\left(\mathbbm{1},\mathbbm{1}\right).\]
\end{thm}
\begin{proof}
Since $\varepsilon$ is invertible, we have the idempotents of Lemma \ref{idempotent-lemma}.
Write $X=\mathcal{M}_n(\delta, \varepsilon)$. Let $I=X(\varepsilon^{-1}\rho_1) + \cdots + X(\varepsilon^{-1}\rho_n)$. Applying \cite[Theorem 4.1]{Boyde} gives the result.
\end{proof}

By combining Theorem \ref{motzkin-thm1} with the results of \cite{BH1}, we immediately obtain a number of results on the vanishing of the homology of Motzkin algebras with invertible parameters.

\begin{cor}
If $\delta \in k$ is also invertible, then $\mathrm{Tor}_{q}^{\mathcal{M}_{n}(\delta, \varepsilon)}\left(\mathbbm{1},\mathbbm{1}\right)=0$ for $q>0$.
\end{cor}
\begin{proof}
This follows from Theorem \ref{motzkin-thm1} and any one of \cite[Theorem A]{BH1}, \cite[Theorem 6.6]{Boyde} or Theorem \ref{G-temp-lieb-thm}.
\end{proof}

\begin{cor}
\label{motzkin-cor2}
Let $\varepsilon \in k$ be invertible. Let $v\in k$ be an invertible element and let $\delta = v + v^{-1}$. Then
\[ \mathrm{Tor}_q^{\mathcal{M}_{n}(\delta, \varepsilon)}\left(\mathbbm{1},\mathbbm{1}\right) =0\]
for $1\leqslant q \leqslant n-2$ when $n$ is even and $1\leqslant q \leqslant n-1$ when $n$ is odd.

In particular, the maps  
\[\mathrm{Tor}_q^{\mathcal{M}_{n-1}(\delta, \varepsilon)}\left(\mathbbm{1},\mathbbm{1}\right)\rightarrow \mathrm{Tor}_q^{\mathcal{M}_{n}(\delta, \varepsilon)}\left(\mathbbm{1},\mathbbm{1}\right)\]
are isomorphisms for $q\leqslant n-3$ so we have homological stability.
\end{cor}
\begin{proof}
This follows from Theorem \ref{motzkin-thm1} and \cite[Theorem B]{BH1}.
\end{proof}

\begin{cor}
\label{motzkin-cor3}
Under the conditions of Corollary \ref{motzkin-cor2}, suppose $n$ is even and that $\delta = v+v^{-1}$ is not invertible. Then
\[\mathrm{Tor}_{n-1}^{\mathcal{M}_{n}(\delta, \varepsilon)}\left(\mathbbm{1},\mathbbm{1}\right)\cong k/\nu k\]
where $\nu$ is a multiple of $\delta$.

In particular, $\mathrm{Tor}_{n-1}^{\mathcal{M}_{n}(\delta, \varepsilon)}\left(\mathbbm{1},\mathbbm{1}\right)\neq 0$.
\end{cor}
\begin{proof}
This follows from Theorem \ref{motzkin-thm1} and \cite[Theorem C]{BH1}.
\end{proof}

Randal-Williams \cite{RW-pre} provided a stronger version of \cite[Theorem B]{BH1}, which yields the following corollary for Motzkin algebras.

\begin{cor}	
Let $\varepsilon \in k$ be invertible and let $\delta \in k$ be any element. Then 
\[ \mathrm{Tor}_q^{\mathcal{M}_{n}(\delta, \varepsilon)}\left(\mathbbm{1},\mathbbm{1}\right) =0\]
for $1\leqslant q \leqslant n-2$ when $n$ is even and $1\leqslant q \leqslant n-1$ when $n$ is odd.
\end{cor}
\begin{proof}
This follows from Theorem \ref{motzkin-thm1} and \cite[Theorem B$^{\prime}$]{RW-pre}.
\end{proof}

\begin{cor}
Let $n$ be odd. Let $k$ be a field whose characteristic does not divide $\binom{j}{t}$ for $1\leqslant t\leqslant j$. Let $\varepsilon \in k$ be invertible and let $\delta=0$. Then
\[ \mathrm{Tor}_q^{\mathcal{M}_{n}(\delta, \varepsilon)}\left(\mathbbm{1},\mathbbm{1}\right) =0\]
for $q>0$.
\end{cor}
\begin{proof}
This follows from Theorem \ref{motzkin-thm1} and \cite[Theorem E]{BH1}.
\end{proof}

\subsection{The equivariant case}

\begin{thm}
\label{motzkin-thm2}
Let $G$ be an abelian group. Let $\varepsilon \in k$ be invertible. Then for any $\delta \in k$ there is an isomorphism of graded $k$-modules
\[\mathrm{Tor}_{\star}^{G\mathcal{M}_{n}(\delta, \varepsilon)}\left(\mathbbm{1},\mathbbm{1}\right) \cong\mathrm{Tor}_{\star}^{G\mathcal{TL}_n(\delta)}\left(\mathbbm{1},\mathbbm{1}\right).\]
\end{thm}
\begin{proof}
Since $\varepsilon$ is invertible, we have the idempotents of Lemma \ref{idempotent-lemma}.
Write $X=G\mathcal{M}_n(\delta, \varepsilon)$. Let $I=X(\varepsilon^{-1}\rho_1) + \cdots + X(\varepsilon^{-1}\rho_n)$. Applying \cite[Theorem 4.1]{Boyde} gives the result.
\end{proof}

\begin{cor}
\label{G-equiv-Motzkin-cor}
If $\delta \in k$ is also invertible then there is an isomorphism of graded $k$-modules
\[\mathrm{Tor}_{\star}^{G\mathcal{M}_{n}(\delta, \varepsilon)}\left(\mathbbm{1},\mathbbm{1}\right) \cong\mathrm{Tor}_{\star}^{k[G^n]}\left(\mathbbm{1},\mathbbm{1}\right).\]
\end{cor}
\begin{proof}
This follows from Theorem \ref{motzkin-thm2} and Theorem \ref{G-temp-lieb-thm}.
\end{proof}

\begin{cor}
Let $G$ be a group. Let $\delta,\,\varepsilon \in k$ be invertible. There is an isomorphism of graded $k$-modules
\[\mathrm{Tor}_{\star}^{G\mathcal{M}_{1}(\delta, \varepsilon)}\left(\mathbbm{1},\mathbbm{1}\right) \cong H_{\star}(G,k).\]
\end{cor}
\begin{proof}
This follows by taking $n=1$ in Corollary \ref{G-equiv-Motzkin-cor}.
\end{proof}

\section{A Sroka-type result for equivariant Brauer algebras}
\label{sroka-sec}
Sroka \cite[Theorem A]{Sroka} has proved a vanishing result for the homology of the Temperley-Lieb algebras when the parameter $\delta$ is not invertible and $n$ is odd. Boyde \cite[Theorem 1.13]{Boyde} extends these methods to prove that the homology of a Brauer algebra is isomorphic to the homology of the symmetric group when the parameter $\delta$ is not invertible and $n$ is odd. He refers to this style of result as a Sroka-type theorem.  

In this section we prove Sroka-type results for equivariant Brauer algebras and equivariant Temperley-Lieb algebras. Throughout this section, let $G$ be an abelian group. We will show that the homology of an equivariant Brauer algebra is isomorphic to the group homology of $G^n\rtimes \Sigma_n$ when $\delta$ is not invertible and $n$ is odd. We will show that the homology of an equivariant Temperley-Lieb algebra is isomorphic to the group homology of $G^n$ when $\delta$ is not invertible and $n$ is odd. 

\subsection{$G$-Brauer algebras with non-invertible parameter}

\begin{defn}
\label{double-diag-defn}
Given two $n$-diagrams $x$ and $y$, the \emph{double diagram} $(x,y)$ consists of three columns of $n$ nodes, with $x$ as the left-hand part of the diagram and $y$ as the right-hand part of the diagram.
\end{defn}

\begin{defn}
Let $(x,y)$ be a double diagram of $G$-Brauer $n$-diagrams. Consider the left-hand column of nodes as being labelled $l_1,\dotsc , l_n$, the middle column of nodes being labelled by $m_1,\dotsc , m_n$ and the right-hand column of nodes being labelled by $r_1,\dotsc , r_n$. For nodes $u,\,v\in \llb  l_1,\dotsc , l_n, m_1,\dotsc , m_n, r_1,\dotsc , r_n\rrb$ we write 
\[u\sim_{xy} v\]
if there is a sequence of edges connecting $u$ and $v$ in the diagram $(x,y)$.
\end{defn}

\begin{defn}
Let $e$ be a diagram and let $p$ be a link state. The \emph{sesqui-diagram} $(p,e)$ has nodes $m_1,\dotsc , m_n$ and $r_1,\dotsc , r_n$, with edges given as follows:
\begin{itemize}
\item $p$ is considered as a right link state and is embedded by mapping its nodes to the vertices $m_1,\dotsc , m_n$;
\item $e$ is embedded by mapping its left nodes to $m_1,\dotsc , m_n$ and its right nodes to $r_1,\dotsc , r_n$. 
\end{itemize}
\end{defn}

\begin{lem}
\label{double-diag-lem1}
Let $p\in GP_i$ be a link state. Suppose that $e\in G\mathcal{B}_n(\delta)$ satisfies the following properties:
\begin{enumerate}
\item the right link state of the underlying unlabelled diagram of $e$ is the underlying unlabelled diagram of $p$;
\item $r_j\sim_{p,e} m_j$ whenever $p$ has a defect at node $j$;
\item for each $j$ there exists some $k$ with $m_j\sim_{p,e} r_k$;
\item each left-to-right connection of $e$ has the label $1\in G$;
\item each right-to-right connection of $e$ has the same label as the corresponding right-to-right connection in $p$;
\item for each right-to-right connection in $p$ with label $g$, there is a left-to-left connection in $e$ with label $g^{-1}$, which shares precisely one vertex with the right-to-right connection.
\end{enumerate}
Then $ye=y$ for all $y\in GJ_p$.
\end{lem}
\begin{proof}
The proof of \cite[Lemma 7.6]{Boyde} ensures that the equality is true for the underlying unlabelled diagrams, using the first three points. Our added conditions ensure that the equality holds for labelled diagrams.
\end{proof}

\begin{lem}
\label{double-diag-lem2}
Let $p\in GP_i$ be a link state with no missing edges and at least one defect. There exists a diagram $f_p\in G\mathcal{B}_n(\delta)$ satisfying the conditions of Lemma \ref{double-diag-lem1}.
\end{lem}
\begin{proof}
The proof of \cite[Lemma 7.7]{Boyde} tells us how to construct an unlabelled diagram satisfying the first three conditions. We satisfy the three labelling conditions by adding identity labels to left-to-right connections, preserving the labels on right-to-right-connections and adding the necessary inverses to left-to-left connections.
\end{proof}

\begin{thm}
\label{non-invertible-equiv-thm}
Let $n$ be a positive integer. Let $I_0$ be the two-sided ideal of $G\mathcal{B}_n(\delta)$ that is free as a $k$-module on diagrams with no left-to-right connections. Then there is an isomorphism of graded $k$-modules
\[\mathrm{Tor}_{\star}^{G\mathcal{B}_n(\delta)/I_0}\left(\mathbbm{1},\mathbbm{1}\right) \cong \mathrm{Tor}_{\star}^{k[G^n\rtimes \Sigma_n]}\left(\mathbbm{1},\mathbbm{1}\right) = H_{\star}\left(G^n\rtimes \Sigma_n , k\right).\]
\end{thm}
\begin{proof}
We apply Proposition \ref{equiv-tech-prop} with $l=1$ and $m=n-1$, noting that $I_{l-1}=I_0$ and
\[ \frac{G\mathcal{B}_n(\delta)}{G\mathcal{B}_n(\delta) \cap I_{n-1}}\cong k[G^n\rtimes \Sigma_n].\]
We need to check that the second hypothesis of Proposition \ref{equiv-tech-prop} holds.

Let $1\leqslant i\leqslant n$. We must take $p\in GP_i$ and find an idempotent generating $G\mathcal{B}_n(\delta) \cap GJ_p$ as a left ideal.

Consider the diagram $f_p$ of Lemma \ref{double-diag-lem2}. By Lemma \ref{double-diag-lem1} with $y=f_p$ we see that $f_p$ is idempotent. 

The ideal generated by $f_p$ is closed under $k$-linear combinations and $G\mathcal{B}_n(\delta)\cap GJ_p$ is free on a basis of diagrams so it suffices to show that for each diagram $y$ in $G\mathcal{B}_n(\delta)\cap GJ_p$, there exists $x\in G\mathcal{B}_n(\delta)$ such that $xf_p=y$. Taking $x$ to be $y$ satisfies this equation by Lemma \ref{double-diag-lem1}. 
\end{proof}

\begin{cor}
Let $n$ be odd. There is an isomorphism of graded $k$-modules
\[\mathrm{Tor}_{\star}^{G\mathcal{B}_n(\delta)}\left(\mathbbm{1},\mathbbm{1}\right) \cong \mathrm{Tor}_{\star}^{k[G^n\rtimes \Sigma_n]}\left(\mathbbm{1},\mathbbm{1}\right) = H_{\star}\left(G^n\rtimes \Sigma_n , k\right).\]
\end{cor}
\begin{proof}
We note that for $n$ odd we have $G\mathcal{B}_n(\delta) \cap I_0=0$. The corollary now follows from Theorem \ref{non-invertible-equiv-thm}.
\end{proof}

\subsection{$G$-Temperley-Lieb algebras with non-invertible parameters}

\begin{defn}
We say that a link state is \emph{planar} if it occurs as the right link state of a $G$-Temperley-Lieb diagram.
\end{defn}

\begin{lem}
\label{sroka-temp-lieb-lem}
Let $p\in GP_i$ be a planar link state with no missing edges and at least one defect. There exists a diagram $e_p\in G\mathcal{TL}_n(\delta)$ satisfying the properties of Lemma \ref{double-diag-lem1}.
\end{lem}
\begin{proof}
The underlying unlabelled diagram is constructed as in \cite[Lemma 7.12]{Boyde}. We add identity labels to all left-to-right connections. We preserve the labels on right-to-right connections and we add the necessary inverses to left-to-left connections.
\end{proof}

\begin{thm}
\label{sroka-temp-lieb-thm}
Let $n$ be a positive integer. Let $I_0$ be the two-sided ideal of $G\mathcal{TL}_n(\delta)$ which is free as a $k$-module on diagrams with no left-to-right connections. Then there is an isomorphism of graded $k$-modules 
\[\mathrm{Tor}_{\star}^{G\mathcal{TL}_n(\delta)/I_0}\left(\mathbbm{1},\mathbbm{1}\right) \cong \mathrm{Tor}_{\star}^{k[G^n]}\left(\mathbbm{1},\mathbbm{1}\right).\]
\end{thm}
\begin{proof}
This follows the same proof as Theorem \ref{non-invertible-equiv-thm}, using Lemma \ref{sroka-temp-lieb-lem} in place of Lemma \ref{double-diag-lem2}.
\end{proof}

\begin{cor}
\label{TL-nodd-cor}
Let $n$ be odd. There is an isomorphism of graded $k$-modules
\[\mathrm{Tor}_{\star}^{G\mathcal{TL}_n(\delta)}\left(\mathbbm{1},\mathbbm{1}\right) \cong \mathrm{Tor}_{\star}^{k[G^n]}\left(\mathbbm{1},\mathbbm{1}\right).\]
\end{cor}
\begin{proof}
We note that for $n$ odd we have $G\mathcal{B}_n(\delta) \cap I_0=0$. The corollary now follows from Theorem \ref{sroka-temp-lieb-thm}.
\end{proof}

\begin{cor}
\label{non-invertible-temp-lieb-cor}
There is an isomorphism of graded $k$-modules
\[\mathrm{Tor}_{\star}^{G\mathcal{TL}_1(\delta)}\left(\mathbbm{1},\mathbbm{1}\right) \cong \mathrm{Tor}_{\star}^{k[G]}\left(\mathbbm{1},\mathbbm{1}\right) = H_{\star}\left(G , k\right).\]
\end{cor}
\begin{proof}
Take $n=1$ in Corollary \ref{TL-nodd-cor} and identify the right-hand side with $H_{\star}(G,k)$.
\end{proof}

\bibliographystyle{alpha}
\bibliography{gen-brauer-refs}

\begin{thebibliography}{CDVDM08}

\bibitem[Abr07]{Abrahamsky}
S.~Abramsky.
\newblock Temperley-lieb algebra: From knot theory to logic and computation via
  quantum mechanics.
\newblock In Goong Chen, Louis Kauffman, and Sam Lomonaco, editors, {\em
  Mathematics of Quantum Computing and Technology}, pages 415--458. Taylor and
  Francis, 2007.

\bibitem[Arn70]{Arnold1}
Vladimir~I. Arnold.
\newblock On some topological invariants of algebraic functions.
\newblock {\em Tr. Mosk. Mat. Obs.}, 21:27--46, 1970.

\bibitem[Arn14]{Arnold2}
Vladimir~I. Arnold.
\newblock {On some topological invariants of algebraic functions}.
\newblock In {\em Vladimir I. Arnold - Collected Works: Hydrodynamics,
  Bifurcation Theory, and Algebraic Geometry 1965-1972}, pages 199--221.
  Springer Berlin Heidelberg, 2014.

\bibitem[BH14]{BHal}
Georgia Benkart and Tom Halverson.
\newblock Motzkin algebras.
\newblock {\em Eur. J. Comb.}, 36:473--502, 2014.

\bibitem[BH22]{BH1}
Rachael Boyd and Richard Hepworth.
\newblock The homology of the {T}emperley-{L}ieb algebras, 2022.
\newblock To appear in Geometry and Topology, arXiv e-print 2006.04256.

\bibitem[BHP21]{BHP}
Rachael Boyd, Richard Hepworth, and Peter Patzt.
\newblock The homology of the {Brauer} algebras.
\newblock {\em Sel. Math., New Ser.}, 27(5):31, 2021.
\newblock Id/No 85.

\bibitem[BHP23]{BHP2}
Rachael Boyd, Richard Hepworth, and Peter Patzt.
\newblock The homology of the partition algebras, 2023.
\newblock arXiv e-print 2303.07979.

\bibitem[BM13]{BM}
Georgia Benkart and Dongho Moon.
\newblock Planar rook algebras and tensor representations of
  {{\(\mathfrak{gl}(1|1)\)}}.
\newblock {\em Commun. Algebra}, 41(7):2405--2416, 2013.

\bibitem[Boy20]{Boyd1}
Rachael Boyd.
\newblock Homological stability for {Artin} monoids.
\newblock {\em Proc. Lond. Math. Soc. (3)}, 121(3):537--583, 2020.

\bibitem[Boy23a]{Boyde-cellular}
Guy Boyde.
\newblock Homological stability and {G}raham-{L}ehrer cellular algebras, 2023.
\newblock arXiv e-print 2310.00373.

\bibitem[Boy23b]{Boyde}
Guy Boyde.
\newblock Idempotents and homology of diagram algebras, 2023.
\newblock arXiv e-print 2212.01826.

\bibitem[Boy23c]{Boyde2}
Guy Boyde.
\newblock Stable homology isomorphisms for the partition and {J}ones annular
  algebras, 2023.
\newblock arXiv e-print 2308.03214.

\bibitem[Bra37]{Brauer}
Richard Brauer.
\newblock On algebras which are connected with the semisimple continuous
  groups.
\newblock {\em Ann. Math. (2)}, 38:857--872, 1937.

\bibitem[BW89]{BW}
Joan~S. Birman and Hans Wenzl.
\newblock Braids, link polynomials and a new algebra.
\newblock {\em Trans. Am. Math. Soc.}, 313(1):249--273, 1989.

\bibitem[Cam24]{Campbell}
John~M. Campbell.
\newblock Alternating submodules for partition algebras, rook algebras, and
  rook-{Brauer} algebras.
\newblock {\em J. Pure Appl. Algebra}, 228(1):31, 2024.
\newblock Id/No 107452.

\bibitem[CDVDM08]{CDDM}
Anton Cox, Maud De~Visscher, Stephen Doty, and Paul Martin.
\newblock On the blocks of the walled {Brauer} algebra.
\newblock {\em J. Algebra}, 320(1):169--212, 2008.

\bibitem[CDVM09a]{CDM}
Anton Cox, Maud De~Visscher, and Paul Martin.
\newblock The blocks of the {Brauer} algebra in characteristic zero.
\newblock {\em Represent. Theory}, 13:272--308, 2009.

\bibitem[CDVM09b]{CDM2}
Anton Cox, Maud De~Visscher, and Paul Martin.
\newblock A geometric characterisation of the blocks of the {Brauer} algebra.
\newblock {\em J. Lond. Math. Soc., II. Ser.}, 80(2):471--494, 2009.

\bibitem[CDVM11]{CDM3}
Anton Cox, Maud De~Visscher, and Paul Martin.
\newblock Alcove geometry and a translation principle for the {Brauer} algebra.
\newblock {\em J. Pure Appl. Algebra}, 215(4):335--367, 2011.

\bibitem[Cha80]{Charney}
Ruth~M. Charney.
\newblock Homology stability for {{\(GL_n\)}} of a {Dedekind} domain.
\newblock {\em Invent. Math.}, 56:1--17, 1980.

\bibitem[DFGG97]{DGG}
P.~Di~Francesco, O.~Golinelli, and E.~Guitter.
\newblock Meanders and the {Temperley}-{Lieb} algebra.
\newblock {\em Commun. Math. Phys.}, 186(1):1--59, 1997.

\bibitem[DVM17]{DM}
Maud De~Visscher and Paul~P. Martin.
\newblock On {Brauer} algebra simple modules over the complex field.
\newblock {\em Trans. Am. Math. Soc.}, 369(3):1579--1609, 2017.

\bibitem[FHH09]{FHH}
Daniel Flath, Tom Halverson, and Kathryn Herbig.
\newblock The planar rook algebra and {Pascal}'s triangle.
\newblock {\em Enseign. Math. (2)}, 55(1-2):77--92, 2009.

\bibitem[GL96]{GL-cellular}
J.~J. Graham and G.~I. Lehrer.
\newblock Cellular algebras.
\newblock {\em Invent. Math.}, 123(1):1--34, 1996.

\bibitem[Gre98]{Green}
R.~M. Green.
\newblock Generalized {Temperley}-{Lieb} algebras and decorated tangles.
\newblock {\em J. Knot Theory Ramifications}, 7(2):155--171, 1998.

\bibitem[Gro02]{Grood}
Cheryl Grood.
\newblock A {Specht} module analog for the rook monoid.
\newblock {\em Electron. J. Comb.}, 9(1):research paper r2, 10, 2002.

\bibitem[GW95]{GW}
Uwe Grimm and S.~Ole Warnaar.
\newblock Solvable {RSOS} models based on the dilute {BWM} algebra.
\newblock {\em Nucl. Phys., B}, 435(3):482--504, 1995.

\bibitem[Hal04]{Halverson}
Tom Halverson.
\newblock Representations of the {{\(q\)}}-rook monoid.
\newblock {\em J. Algebra}, 273(1):227--251, 2004.

\bibitem[Har85]{Harer}
John~L. Harer.
\newblock Stability of the homology of the mapping class groups of orientable
  surfaces.
\newblock {\em Ann. Math. (2)}, 121:215--249, 1985.

\bibitem[Hd14]{HD}
Tom Halverson and Elise delMas.
\newblock Representations of the rook-{Brauer} algebra.
\newblock {\em Commun. Algebra}, 42(1):423--443, 2014.

\bibitem[Hep16]{Hepworth2}
Richard Hepworth.
\newblock Homological stability for families of {Coxeter} groups.
\newblock {\em Algebr. Geom. Topol.}, 16(5):2779--2811, 2016.

\bibitem[Hep22]{Hepworth}
Richard Hepworth.
\newblock Homological stability for {Iwahori}-{Hecke} algebras.
\newblock {\em J. Topol.}, 15(4):2174--2215, 2022.

\bibitem[HR01]{HalvR}
T.~Halverson and A.~Ram.
\newblock {{\(q\)}}-rook monoid algebras, {Hecke} algebras, and {Schur}-{Weyl}
  duality.
\newblock {\em J. Math. Sci., New York}, 121(3):2419--2436, 2001.

\bibitem[HV98]{HV1}
Allen Hatcher and Karen Vogtmann.
\newblock Cerf theory for graphs.
\newblock {\em J. Lond. Math. Soc., II. Ser.}, 58(3):633--655, 1998.

\bibitem[HV04]{HV}
Allen Hatcher and Karen Vogtmann.
\newblock Homology stability for outer automorphism groups of free groups.
\newblock {\em Algebr. Geom. Topol.}, 4:1253--1272, 2004.

\bibitem[HVW06]{HVW}
Allen Hatcher, Karen Vogtmann, and Nathalie Wahl.
\newblock Erratum to: {Homology} stability for outer automorphism groups of
  free groups.
\newblock {\em Algebr. Geom. Topol.}, 6:573--579, 2006.

\bibitem[HW05]{HW2}
Allen Hatcher and Nathalie Wahl.
\newblock Stabilization for the automorphisms of free groups with boundaries.
\newblock {\em Geom. Topol.}, 9:1295--1336, 2005.

\bibitem[HW10]{HW}
Allen Hatcher and Nathalie Wahl.
\newblock Stabilization for mapping class groups of 3-manifolds.
\newblock {\em Duke Math. J.}, 155(2):205--269, 2010.

\bibitem[Jon83]{Jones1}
V.~F.~R. Jones.
\newblock Index for subfactors.
\newblock {\em Invent. Math.}, 72:1--25, 1983.

\bibitem[Jon85]{Jones2}
Vaughan F.~R. Jones.
\newblock A polynomial invariant for knots via von {Neumann} algebras.
\newblock {\em Bull. Am. Math. Soc., New Ser.}, 12:103--111, 1985.

\bibitem[JY21]{JoYa}
Vaughan F.~R. Jones and Jun Yang.
\newblock Motzkin algebras and the {{\(A_n\)}} tensor categories of bimodules.
\newblock {\em Int. J. Math.}, 32(10):43, 2021.
\newblock Id/No 2150077.

\bibitem[Kau90]{Kauffmann1}
Louis~H. Kauffman.
\newblock An invariant of regular isotopy.
\newblock {\em Trans. Am. Math. Soc.}, 318(2):417--471, 1990.

\bibitem[KM06]{KM}
Ganna Kudryavtseva and Volodymyr Mazorchuk.
\newblock On presentations of {Brauer}-type monoids.
\newblock {\em Cent. Eur. J. Math.}, 4(3):413--434, 2006.

\bibitem[KMP18]{KMP}
O.~H. King, P.~P. Martin, and A.~E. Parker.
\newblock On central idempotents in the {Brauer} algebra.
\newblock {\em J. Algebra}, 512:20--46, 2018.

\bibitem[Mar87]{MartinTL1}
P.~P. Martin.
\newblock All direct sum representations of the {Temperley}-{Lieb} algebra.
\newblock {\em J. Phys. A, Math. Gen.}, 20:l539--l542, 1987.

\bibitem[Mar88]{MartinTL2}
P.~P. Martin.
\newblock Temperley-{Lieb} algebra, group theory and the {Potts} model.
\newblock {\em J. Phys. A, Math. Gen.}, 21(3):577--591, 1988.

\bibitem[Mar90a]{MartinTL3}
P.~P. Martin.
\newblock Temperley-{Lieb} algebras and the long distance properties of
  statistical mechanical models.
\newblock {\em J. Phys. A, Math. Gen.}, 23(1):7--30, 1990.

\bibitem[Mar90b]{MartinTL4}
Paul~P. Martin.
\newblock Representations of graph {Temperley}-{Lieb} algebras.
\newblock {\em Publ. Res. Inst. Math. Sci.}, 26(3):485--503, 1990.

\bibitem[Mar94]{MartinTL}
Paul Martin.
\newblock Temperley-{Lieb} algebras for non-planar statistical mechanics. --
  {The} partition algebra construction.
\newblock {\em J. Knot Theory Ramifications}, 3(1):51--82, 1994.

\bibitem[Mar15]{Martin2}
Paul~P. Martin.
\newblock The decomposition matrices of the {Brauer} algebra over the complex
  field.
\newblock {\em Trans. Am. Math. Soc.}, 367(3):1797--1825, 2015.

\bibitem[MM14]{MM}
Paul Martin and Volodymyr Mazorchuk.
\newblock On the representation theory of partial {Brauer} algebras.
\newblock {\em Q. J. Math.}, 65(1):225--247, 2014.

\bibitem[Mos22]{Moselle}
Isaac Moselle.
\newblock Homological stability for {I}wahori-{H}ecke algebras of type b, 2022.
\newblock arXiv e-print 2211.06230.

\bibitem[Mot48]{Motzkin}
Th. Motzkin.
\newblock Relations between hypersurface corss ratios, and a combinatorial
  formula for partitions of a polygon, for permanent preponderance, and for
  non- associative products.
\newblock {\em Bull. Am. Math. Soc.}, 54:352--360, 1948.

\bibitem[MT90]{MT}
H.~R. Morton and P.~Traczyk.
\newblock Knots and algebras.
\newblock In {\em Contribuciones matem\'aticas en homenaje al profesor D.
  Antonio Plans Sanz de Bremond}, pages 201--219. Zaragoza: Secretariado de
  Publicaciones, Universidad de Zaragoza, 1990.

\bibitem[Mur87]{M}
Jun Murakami.
\newblock The {Kauffman} polynomial of links and representation theory.
\newblock {\em Osaka J. Math.}, 24:745--758, 1987.

\bibitem[Nak60]{Nakaoka}
Minoru Nakaoka.
\newblock Decomposition theorem for homology groups of symmetric groups.
\newblock {\em Ann. Math. (2)}, 71:16--42, 1960.

\bibitem[RSA14]{RSA}
David Ridout and Yvan Saint-Aubin.
\newblock Standard modules, induction and the structure of the
  {Temperley}-{Lieb} algebra.
\newblock {\em Adv. Theor. Math. Phys.}, 18(5):957--1041, 2014.

\bibitem[RW21]{RW-pre}
Oscar Randal-Williams.
\newblock A remark on the homology of {T}emperley-{L}ieb algebras, 2021.
\newblock \url{https://www.dpmms.cam.ac.uk/~or257/notes/Descent.pdf}.

\bibitem[RWW17]{RWW}
Oscar Randal-Williams and Nathalie Wahl.
\newblock Homological stability for automorphism groups.
\newblock {\em Adv. Math.}, 318:534--626, 2017.

\bibitem[Sol02]{Solomon}
Louis Solomon.
\newblock Representations of the rook monoid.
\newblock {\em J. Algebra}, 256(2):309--342, 2002.

\bibitem[Sro22]{Sroka}
Robin Sroka.
\newblock The homology of a {T}emperley-{L}ieb algebra on an odd number of
  strands, 2022.
\newblock To appear in Algebraic \& Geometric Toplogy. arXiv e-print
  2202.08799.

\bibitem[SW19]{SW}
Markus Szymik and Nathalie Wahl.
\newblock The homology of the {Higman}-{Thompson} groups.
\newblock {\em Invent. Math.}, 216(2):445--518, 2019.

\bibitem[Til16]{Tillmann}
Ulrike Tillmann.
\newblock Homology stability for symmetric diffeomorphism and mapping class
  groups.
\newblock {\em Math. Proc. Camb. Philos. Soc.}, 160(1):121--139, 2016.

\bibitem[TL71]{TL}
H.~V.~N. Temperley and Elliott~H. Lieb.
\newblock Relations between the `percolation' and `colouring' problem and other
  graph-theoretical problems associated with regular planar lattices: some
  exact results for the 'percolation' problem.
\newblock {\em Proc. R. Soc. Lond., Ser. A}, 322:251--280, 1971.

\bibitem[vdK80]{VDK}
Wilberd van~der Kallen.
\newblock Homology stability for linear groups.
\newblock {\em Invent. Math.}, 60:269--295, 1980.

\bibitem[Wes95]{Westbury}
B.~W. Westbury.
\newblock The representation theory of the {Temperley}-{Lieb} algebras.
\newblock {\em Math. Z.}, 219(4):539--565, 1995.

\bibitem[Xia16]{Xiao}
Zhan~Kui Xiao.
\newblock On tensor spaces for rook monoid algebras.
\newblock {\em Acta Math. Sin., Engl. Ser.}, 32(5):607--620, 2016.

\end{thebibliography}

\end{document}